\documentclass[pdflatex,sn-mathphys-num]{sn-jnl}

\usepackage{graphicx}%
\usepackage{multirow}%
\usepackage{amsmath,amssymb,amsfonts}%
\usepackage{amsthm}%
\usepackage{mathrsfs}%
\usepackage[title]{appendix}%
\usepackage{xcolor}%
\usepackage{textcomp}%
\usepackage{manyfoot}%
\usepackage{booktabs}%
\usepackage{algorithm}%
\usepackage{algorithmicx}%
\usepackage{algpseudocode}%
\usepackage{listings}%
\usepackage{todonotes}
\usepackage{todonotes}
\usepackage[framemethod=tikz]{mdframed}
\mdfsetup{%
  linewidth=1.25pt,
  userdefinedwidth=\textwidth,
  roundcorner=10pt,
}

\theoremstyle{thmstyleone}%
\newtheorem{theorem}{Theorem}[section]
\newtheorem{proposition}{Proposition}[section]%
\newtheorem{assumption}{Assumption}[section]
\newtheorem{corollary}{Corollary}[section]
\newtheorem{lemma}{Lemma}[section]

\theoremstyle{thmstyletwo}%
\newtheorem{remark}{Remark}[section]%

\theoremstyle{thmstylethree}%
\newtheorem{definition}{Definition}[section]%
\newtheorem{property}{Property}[section]

\usepackage{color}
\usepackage{hyperref}
\usepackage{multirow}
\usepackage{mathtools}
\usepackage{enumerate}
\def\R{{\mathbb{R}}}
\def\N{{\mathbb{N}}}

\def\dist{{\rm dist}}
\def\dom{\mathop{\rm dom}}

\def\argmin{\mathop{\rm argmin}}

\numberwithin{equation}{section}

\begin{document}

\title[KL-based convergence analysis with approximate weakly convex subdifferentials]
{Kurdyka–{\L}ojasiewicz-based convergence analysis with approximate weakly convex subdifferentials: bundle, quadratic penalty, and moving balls approximation methods}


	\author[1]{\fnm{Jiefeng} \sur{Xu}}\email{jiefeng.xu@polyu.edu.hk}

    \author*[2]{\fnm{Guoyin} \sur{Li}}\email{g.li@unsw.edu.au}
	
	\author[1]{\fnm{Ting Kei} \sur{Pong}}\email{tk.pong@polyu.edu.hk}

	\affil[1]{\orgdiv{Department of Applied Mathematics}, \orgname{the Hong Kong Polytechnic University}, \city{Hong Kong}, \country{China}}
	
	\affil[2]{\orgdiv{Department of Applied Mathematics}, \orgname{University of New South Wales}, \city{Sydney}, \country{Australia}}


\abstract{The Kurdyka–{\L}ojasiewicz (KL) property has been widely used for establishing global sequential convergence and analyzing local convergence rate of optimization algorithms. Most existing analysis frameworks rely on the construction of a suitable KL potential function such that along the sequence generated, it admits sufficient descents and its {\em exact} subdifferentials are suitably controlled by successive changes of the iterates. 
Nevertheless, a handful of important algorithms, such as bundle methods and quadratic penalty methods, are not known to have potential functions verifying the aforementioned properties, rendering their convergence analysis out of the reach of existing KL-based convergence analysis template. In this paper, we develop a new convergence analysis framework to extend the applicability of KL-based analysis. Our framework is built on a weakly convex KL potential function $\psi$. It allows {\em approximate} sufficient descents, the use of {\em diminishing} stepsizes, and controls on {\em enlarged} subdifferentials of $\psi$ known as the $\varrho$-weakly-convex-$\varepsilon$-subdifferentials. We illustrate the usefulness of our framework by establishing the global sequential convergence and analyzing the local convergence rate of some recent variants of bundle method, quadratic penalty method and moving
balls approximation method. Specifically, we show that these methods fall under our framework by presenting {\em explicit} and {\em natural} choices of $\psi$, and derive suitable choices of model parameters, penalty parameters {and constraint shrinkage parameters}, respectively, for establishing their global sequential convergence and local convergence rate based on the KL property and exponent of $\psi$. }




\maketitle

\section{Introduction}



\medskip 

The Kurdyka–{\L}ojasiewicz (KL) inequality and the associated KL convergence framework have become indispensable tools for establishing the convergence of iterative algorithms in modern nonconvex and nonsmooth optimization.  Since the seminal developments of \cite{10AttouchBolteRedontSoubeyran,13AttouchBolteSvaiter,14BolteSabachTeboulle} which insulate the important role of KL inequality in analyzing convergence of descent methods, the framework has been successfully applied to a wide range of optimization algorithms, including gradient-type, proximal, block-coordinate, splitting, alternating minimization methods and their accelerated variants; see, e.g., \cite{10AttouchBolteRedontSoubeyran,19Ochs,22BotDaoLi,16LP,13XuYin,09AttouchBolte}. 

\medskip 

A central feature of KL analysis framework, as discussed in \cite{10AttouchBolteRedontSoubeyran,13AttouchBolteSvaiter,14BolteSabachTeboulle}, is that convergence is obtained from combining\vspace{-0.15 cm}
\begin{enumerate}[({H}1)]
    \item a sufficient descent property;
    \item a suitable control of the stationarity measure, expressed in terms of the {\em exact} subdifferential at the current iterate, by successive changes of the iterates; 
    \item a suitable continuity or limit-point condition.
\end{enumerate}  \vspace{-0.15 cm}
Upon verifying these conditions on a suitable (typically algorithm-specific) potential function and supposing that such a function satisfies the KL property (i.e., the KL inequality holds at every point in the domain of its subdifferential), then one can use the KL inequality to establish convergence of the iterates and, in many cases, to derive explicit convergence rates; see, e.g., \cite{10AttouchBolteRedontSoubeyran,14BolteSabachTeboulle,18LP,19Ochs,22BotDaoLi}. For example, in \cite{10AttouchBolteRedontSoubeyran,13AttouchBolteSvaiter, 14BolteSabachTeboulle}, the authors  laid the foundations for using the KL inequality as a general framework for analyzing proximal-type algorithms under assumptions (H1)–(H3), especially for nonconvex and nonsmooth optimization problems with sum-structured objective functions. Subsequently, a flexible framework was developed in \cite{19Ochs} based on bivariate merit functions, allowing for the treatment of inertial methods. This framework was later extended to accommodate nonsmooth, nonconvex fractional potential functions \cite{22BotDaoLi,YKZH25}, and to handle optimization algorithms with inexactly solved subproblems \cite{21Sun}. To obtain explicit convergence rates from the KL inequality, it is often necessary to know the associated exponent, commonly referred to as the KL exponent. In this direction, some calculus rules for KL exponents were established in \cite{18LP,22Peiran}, which enable KL exponents to be derived from simpler settings, and demonstrated their application to the convergence-rate analysis of first-order methods.  For further recent developments, see \cite{25Boris_Nesterov}.

\medskip

While existing KL-based convergence analyses are well established for a broad class of descent methods, their direct applicability to several important algorithms remains limited when the standard assumptions (H1)–(H3) fail. This limitation arises in particular from algorithm-dependent approximation mechanisms where iterates may be generated using bundle models, variable smoothing techniques, or penalty functions with varying penalty parameters, etc.; see, e.g., \cite{BW21,26LMX,25BDL,ASSS23}. 
Indeed, for many of these algorithms, either a potential function verifying (H1) is unknown, or the stationarity measure based on {\em exact subdifferential} of the potential function cannot be readily/suitably controlled as required in (H2), rendering the analysis of these algorithms out of the reach of standard KL-based convergence analysis template. 
It is also worth noting that non-diminishing stepsizes, including constant stepsizes, are a natural feature of many descent-based methods satisfying (H1)--(H3). On the other hand, diminishing stepsizes arise naturally in several important settings, including subgradient, stochastic approximation and penalty-type methods (see, e.g., \cite{23LMQ,25QiuJiangMilzarek,Robinson99,26LMX,24AlacaogluWright}), where nonsmoothness, inexactness, constraint violation, or stochastic errors may preclude the use of non-diminishing stepsizes. 

\medskip
Several recent works have relaxed or modified (H1)–(H2) to accommodate specific numerical methods, including nonmonotone descent methods with non-diminishing stepsizes \cite{23QianPan}, some proximal-gradient-type algorithms with diminishing stepsizes \cite{25QiuJiangMilzarek,24QMLM} (under suitable additional conditions), and a variable-smoothing algorithm for specific convex composite and fractional structures \cite{26Tao}. These frameworks allow approximate descent and / or relaxed bounds on the {\em exact} subdifferential of the potential function. Nevertheless, it is unclear whether these frameworks can be directly applied or readily adapted to analyze methods involving bundle models or penalty functions with varying penalty parameters. One major difficulty, as mentioned before, is the construction of a potential function whose {\em exact} subdifferential along the iterates can be suitably controlled. Indeed, some known potential functions are highly algorithm-specific, especially for algorithms with subproblems constructed via sophisticated approximation techniques; see, e.g., the potential functions used in  \cite{20BCP,16BP,21YPL} for analyzing variants of the moving balls approximation method.


\medskip

In this paper, we develop a unified and flexible framework that can accommodate diverse algorithmic settings (including the use of diminishing stepsizes) by allowing further {\em controlled} violations of (H1) and (H2), potentially easing the difficulty in the search / construction of a suitable potential function. Specifically, our framework is based on a {\em weakly convex} potential function $\psi$ and we replace the stationarity measure in (H2) with one that is based on the $\varrho$-weakly-convex-$\varepsilon$-subdifferential of $\psi$; see (C2) in Section~\ref{sec_applications}. With suitable choices of algorithmic parameters, we establish global sequential convergence based on the KL property of $\psi$, and deduce the local convergence rate based on the KL exponent of $\psi$. Our framework can be applied to the following four classes of algorithms (see Sections~\ref{sec3.1}, \ref{sec3.2}, \ref{sec3.3}, and \ref{sec3.4}, respectively): 
\vspace{-0.346cm}
\begin{enumerate}[{\bf {M}1}]
 \item a class of bundle methods for weakly convex optimization \cite{ASSS23};
 \item a quadratic (smooth) penalty method for minimizing a weakly convex objective 
 over the intersection of a closed convex set and smooth equality constraints \cite{26LMX};
 \item the inexact level constrained proximal gradient method for minimizing a weakly convex objective over (possibly multiple) weakly convex inequality constraints \cite{25BDL}, which is a form of moving balls approximation method;
 \item a variable smoothing method for weakly convex optimization \cite{BW21}.
\end{enumerate}

One  notable feature of our framework is that the potential function $\psi$ often takes a natural and simple form.
Specifically, the potential function $\psi$ in our framework can be taken to be just the {\em original objective function} for {\bf M1} and {\bf M4}, while for {\bf M2} and {\bf M3} that involve explicit equality / inequality constraints, we take $\psi$ naturally as the {\em classical exact penalty function}.
To the best of our knowledge, for {\bf M1} and {\bf M2}, our results are the first global sequential convergence and local convergence rate results based on KL property, while for {\bf M3} and {\bf M4}, our results are the first global sequential convergence and local convergence rate results in the nonconvex setting that are derived based on the KL property of the {\em classical exact penalty function} and the {\em original objective}, respectively. Note that global sequential convergence and convergence rate results for some variants of {\bf M3} were previously analyzed based on the KL property of a special potential function constructed based on the subproblems involved \cite{20BCP,16BP,21YPL}, while global sequential convergence for a variant of {\bf M4} has been established very recently in \cite{26Tao} based on the KL property of a special potential function involving the Moreau envelope and its epigraphical closure.

\medskip

Our use of $\varrho$-weakly-convex-$\varepsilon$-subdifferential of $\psi$ is motivated by convergence analysis of subgradient and bundle-type methods in recent works \cite{24AAS,Robinson99,ASSS23}, which exploited $\varrho$-weakly-convex-$\varepsilon$-subdifferential and metric subregularity / error bound conditions to deduce local linear convergence of their methods. Unlike their work, our analysis is based on the KL property, and we also derive local convergence rates (linear or sublinear convergence) according to the value of the KL exponent. Our framework allows the use of diminishing stepsizes (see (C1) in Section~\ref{sec_applications}); although our framework is designed for (possibly) nonsmooth deterministic optimization problems, it is interesting to note that our conditions on valid stepsizes / parameters (see Property~\ref{ass-param}, especially \eqref{eq-sum-bdd}) bear resemblance to those used in recent works that developed KL-based convergence analysis for some proximal-gradient-type algorithms with diminishing stepsizes; see, e.g., \cite[Eq.~(4.3)]{25QiuJiangMilzarek} and \cite[Section~2.2~(A.4)]{24QMLM}. On the other hand, these existing works are primarily motivated by stochastic
proximal-gradient-type algorithms, and their analyses do not directly
extend to the four classes of algorithms considered
above. 

\medskip

The rest of the paper is organized as follows. We present notation and preliminaries in Section~\ref{sec2}. Our algorithmic framework is presented in Section~\ref{sec_applications}, and the discussions concerning the aforementioned four classes of algorithms are presented in Sections~\ref{sec3.1}, \ref{sec3.2}, \ref{sec3.3} and \ref{sec3.4}, respectively. Finally, we present the convergence analysis of the sequence generated by our algorithmic framework in Section~\ref{sec4}.

\section{Preliminary and notation}\label{sec2}
In this paper, we let $\mathbb{X}$ be a finite dimensional real Hilbert space with inner product $\langle\cdot,\cdot\rangle$ and induced norm $\|\cdot\|$. 
We also let $\mathbb{N}_0$ and $\mathbb{N}$ be the sets of nonnegative and positive integers, respectively, and write 
$[m]\coloneqq \{1,\ldots,m\}$ for any $m\in \mathbb{N}$.

We say that an extended-real-valued function $h:\mathbb{X}\to[-\infty,\infty]$ is \emph{proper} if $h(x)>-\infty$ for every $x\in\mathbb{X}$ and its domain
$\dom h\coloneqq\{x\in\mathbb{X}: h(x)<\infty\}$ is nonempty. A proper function $h$ is said to be \emph{closed} if it is lower semicontinuous.
For a proper closed function $h$, the regular (Fréchet) subdifferential and the (limiting) subdifferential of $h$ at an $x\in\dom h$ are defined, respectively, by
\begin{align}
& \widehat{\partial} h(x) \coloneqq \left\{ 
\xi \in\mathbb{X} : \liminf_{\substack{y\to x, y\neq x}} \frac{h(y)-h(x)-\langle \xi,y-x\rangle}{\|y-x\|} \ge 0
\right\}, \notag
\\ & \partial h(x) \coloneqq \left\{
\xi \in\mathbb{X}: \exists\, x^k\xrightarrow{h} x,\; \xi^k\to \xi 
\text{ with } \xi^k\in\widehat{\partial}h(x^k)\text{ for each }k
\right\}, \notag 
\end{align}
where $x^k\xrightarrow{h}x$ means $x^k\to x$ and $h(x^k)\to h(x)$. The above subdifferentials are set to $\emptyset$ if $x\notin \dom h$.
For a proper closed convex function $h$, the above two subdifferentials
coincide with the classical convex subdifferential, i.e., at any $x\in \mathbb{X}$, 
\[
\partial h(x) = \widehat{\partial} h(x) = \bigl\{ \xi \in \mathbb{X} : h(y) \ge h(x) + \langle \xi, y-x\rangle
\ \ \ \forall\, y \in \mathbb{X} \bigr\};
\]
see, e.g., \cite[Proposition 8.12]{RW98}.
For $\varepsilon \ge 0$, the \emph{$\varepsilon$-subdifferential} of a proper closed convex
function $h$ at an $x\in \mathbb{X}$ is defined as
\[
\partial_\varepsilon h(x) \coloneqq \bigl\{ \xi \in \mathbb{X} : h(y) \ge h(x) +
\langle \xi, y-x\rangle - \varepsilon \ \ \ \forall\, y \in \mathbb{X} \bigr\}.
\]
Clearly, $\partial_0 h(x) = \partial h(x)$, and $\partial_{\varepsilon_1} h(x)\subseteq \partial_{\varepsilon_2} h(x)$
whenever $0\le \varepsilon_1\le\varepsilon_2$.

A proper closed function
$h:\mathbb{X}\to(-\infty,\infty]$ is said to be
$\varrho$-weakly convex for some constant $\varrho\ge0$ if
$h+\frac{\varrho}{2}\|\cdot\|^2$ is convex.
By \cite[Exercise 8.8]{RW98}, 
for every $\varrho$-weakly convex function $h$,
we have for any $x\in \mathbb{X}$ that,
\begin{align}
\partial h(x) & = \partial\!\left(h+\tfrac{\varrho}{2}\|\cdot\|^2\right)(x)
-\varrho x \notag
\\ & = \left\{\xi \in \mathbb{X} : h(y)  
\ge h(x) + \langle \xi, y-x\rangle 
- \tfrac{\varrho}{2}\|y-x\|^2\ \ \ \forall y \in \mathbb{X}\right\}.
\label{eq-wc-subd}
\end{align}
For a $\varrho$-weakly convex function $h$ and a constant $\varepsilon\ge0$, the $\varrho$-weakly-convex-$\varepsilon$-subdifferential ($\varrho$-w.c.-$\varepsilon$-subdifferential) of $h$ at an $x\in \mathbb{X}$ is defined as
\begin{align}
    \partial_{\varepsilon}^{\varrho} h(x)
&  \coloneqq \partial_{\varepsilon} (h + \tfrac{\varrho}{2}\|\cdot\|^2)(x)
- \varrho x \notag
\\ & = \left\{\xi \in \mathbb{X} : h(y)  
\ge h(x) + \langle \xi, y-x\rangle 
- \tfrac{\varrho}{2}\|y-x\|^2 - \varepsilon\ \ \ \forall y \in \mathbb{X}\right\}. 
\label{eq-eps}
\end{align}

Now, we recall the Kurdyka-{\L}ojasiewicz (KL) property \cite{10AttouchBolteRedontSoubeyran,13AttouchBolteSvaiter, 14BolteSabachTeboulle}. 
For notational simplicity, we let $\mathcal{Q}_{\eta}$ with $\eta \in (0, \infty]$ be the set 
of all concave continuous functions $\varphi: [0, \eta)\to [0, \infty)$
that are continuously differentiable on $(0, \eta)$ with $\varphi^{\prime}>0$
and $\varphi(0) = 0$.
\begin{definition}[KL property]
A proper closed function $h$ is said to satisfy the KL property at $\bar{x}\in 
\dom  \partial h$ if there exist $\eta \in (0, \infty]$, 
$\delta>0$, and $\varphi \in \mathcal{Q}_{\eta}$ such that
\begin{equation*}
\varphi^{\prime}(h(x) - h(\bar{x}))\mathrm{dist} (0, \partial h(x)) \ge 1
\end{equation*}
whenever $x \in \dom  \partial h$ with $\|x - \bar{x}\| \le \delta$ and 
$h(\bar{x}) < h(x) < h(\bar{x}) + \eta$.
A proper closed function $h$ satisfying the KL property at all points in $\dom \partial h$ is called a KL function.
\end{definition}

\begin{definition}[KL exponent]
A proper closed function $h$ is said to satisfy the KL property at $\bar{x}\in 
\dom  \partial h$ with an exponent $\theta\in [0, 1)$
if it satisfies the KL property with $\varphi(t) \coloneqq \kappa t^{1 - \theta}$ for some constant $\kappa>0$, i.e., 
there exist $\eta \in (0, \infty]$, $\delta>0$ and $\widetilde{\kappa}>0$ such that, 
\begin{equation*}
\mathrm{dist} (0, \partial h(x)) 
\ge \widetilde{\kappa}(h(x) - h(\bar{x}))^{\theta}
\end{equation*}
whenever $x \in \dom  \partial h$ with $\|x - \bar{x}\| \le \delta$ and 
$h(\bar{x}) < h(x) < h(\bar{x}) + \eta$.
If $h$ is a KL function and has the same exponent $\theta \in [0,1)$ at all points in $\dom \partial h$, we call $h$ a KL function with exponent $\theta$.
\end{definition}

The KL property is a standard tool for establishing global sequential convergence for algorithms that satisfy conditions (H1), (H2) and (H3) for some nonconvex optimization problems, while the KL exponent is used to derive local convergence rates. As in \cite{14BolteSabachTeboulle}, we will also use the following uniformized KL property to simplify the convergence analysis.

\begin{lemma}[{\cite[Lemma 6]{14BolteSabachTeboulle}, Uniformized KL property}]
\label{lm-KL-unif}
Let $h: \mathbb{X} \rightarrow (-\infty, \infty]$ be a proper closed function and $\Omega$ be a compact subset of $\dom \partial h$. Suppose that $h$ is constant on $\Omega$ and satisfies the KL property at each point of $\Omega$. Then there exist $\delta>0$, $\eta\in (0, \infty]$ and $\varphi \in \mathcal{Q}_\eta$ such that
\begin{equation*}
\varphi^{\prime}(h(x)-h(\bar{x})) \operatorname{dist}(0, \partial h(x)) \ge 1
\end{equation*}
for all $\bar x \in \Omega$ and $x\in \dom \partial h$ satisfying $\mathrm{dist}(x, \Omega) \le \delta$ and 
$h(\bar{x}) < h(x) < h(\bar{x}) + \eta$.
\end{lemma}

KL-based convergence analyses for algorithms allowing approximate sufficient descent and inexact solutions of subproblems have been recently investigated; see, e.g., \cite{YKZH25,23HuLiu,21Sun}. In particular, the following auxiliary function was used in \cite{21Sun,YKZH25} as a potential function to obtain global sequential convergence for the corresponding algorithms, under the KL analysis framework: \begin{equation}\label{eq-F-tau-def}
    \Psi_{\tau}(x,t)
    \coloneqq \psi(x)+|t|^{\tau}\ \ \ \ \forall   (x,t)\in\mathbb{X}\times\mathbb{R},
\end{equation}
where $\tau>1$ is a constant and $\psi:\mathbb{X}\to (-\infty, \infty]$ is a proper closed function.

We establish in the next proposition a relationship between the KL properties of $\psi$ and $\Psi_{\tau}$. This generalizes the corresponding result concerning their KL exponents in \cite[Proposition~4.2]{YKZH25}. The proof is given in Appendix~\ref{AppendixA}.

\begin{proposition}\label{prop-kl}
	Let $\mathbb{X}$ and $\mathbb{Y}$ be two finite-dimensional real Hilbert spaces.
	Let $f:\mathbb{X}\to (-\infty, \infty]$ and $g:\mathbb{Y}\to \mathbb{R}$ be two proper closed functions.
	Define 
	$$
	\Psi(x, y) \coloneqq f(x) + g(y) \quad \ \forall x\in \mathbb{X}, y\in \mathbb{Y}.
	$$
	Suppose that $f$ and $g$ satisfy the KL property at 
	$\bar x \in \dom \partial f$ and $\bar{y} \in \dom \partial g$, respectively.
    If $g$ is continuous at $\bar y$, then $(\bar x,\bar y)\in \dom \partial \Psi$ and $\Psi$ satisfies the KL property at $(\bar x, \bar y)$.
\end{proposition}

Applying Proposition~\ref{prop-kl} to $\Psi_\tau$ in \eqref{eq-F-tau-def} gives the following immediate corollary.
\begin{corollary}
		\label{cor-kl}
        Let $\psi:\mathbb{X}\to (-\infty, \infty]$ be a proper closed function.
		Suppose that $\psi$ satisfies the KL property at $\bar{x} \in {\rm dom}\, \partial \psi$. Then, for any $\tau>1$, for the $\Psi_\tau$ defined in \eqref{eq-F-tau-def}, it holds that $(\bar x,0)\in \dom \partial \Psi_\tau$ and $\Psi_\tau$ satisfies the KL property at $(\bar{x}, 0)$.	
\end{corollary}

We next recall a Br\o{}ndsted--Rockafellar-type relation.

\begin{theorem}[{Br\o{}ndsted-Rockafellar-type  relation; \cite[Theorem~2]{Robinson99}}]\label{th-BRL}
Let $h: \mathbb{X} \to (-\infty, \infty]$ be a proper closed convex function. 
Let $x\in \mathbb{X}$ and $\xi \in \partial_{\varepsilon}h(x)$ for some $\varepsilon\ge 0$.
Then, for each $\gamma>0$, there is a unique $y = y(\gamma)$ such that
\begin{equation*}
\xi - \tfrac{1}{\gamma}y \in \partial h(x + \gamma y)
\quad \mbox{and} \quad 
\|y\| \le \sqrt{\varepsilon}.
\end{equation*}
\end{theorem}

The following lemma establishes some properties of $\varrho$-w.c.-$\varepsilon$-subdifferentials of weakly convex functions and is analogous to \cite[Theorem~3.1]{24AAS}.
Specifically, we identify an element $\zeta\in \partial h(z)$ in \eqref{eq-x-z-zeta} whose deviation from the given $\xi$ is proportional to $x - z$. Moreover, compared with the last relation in \cite[Theorem~3.1]{24AAS}, the term involving $\epsilon^{\frac12}$ in \eqref{eq-f-x-z} scales with $\|\xi\|$ rather than $\|\xi\|+\kappa$ for some $\kappa > 0$. These relations will play a key role in our application of the KL property to the convergence analysis of our framework in Section~\ref{sec_applications}.

\begin{lemma}\label{lm-BR}
Let $\varrho$ and $\varepsilon$ be two nonnegative numbers.
Let $h: \mathbb{X} \rightarrow (-\infty, \infty]$ be a proper closed $\varrho$-weakly convex function and $x\in \mathbb{X}$.
If $\xi \in \partial_{\varepsilon}^{\varrho} h(x)$, then
there exists $z\in \mathbb{X}$
such that
\begin{align}
& \|x - z\| \le \varepsilon^{\frac{1}{2}}, \quad
\zeta \coloneqq  \xi +  (1 + \varrho) (x - z) \in \partial h(z), \label{eq-x-z-zeta}
\\ & |h(x) - h(z)| \le \varepsilon^{\frac{1}{2}}\|\xi\|
+ (1 + \tfrac{\varrho}{2})\varepsilon.
\label{eq-f-x-z}
\end{align}
\end{lemma}
\begin{proof}
From $\xi \in \partial_{\varepsilon}^{\varrho} h(x)$, we deduce from \eqref{eq-eps} that
$\xi + \varrho x \in \partial_{\varepsilon} (h + \tfrac{\varrho}{2}\|\cdot\|^2)(x)$.
Then, by Theorem~\ref{th-BRL} with $\gamma=1$, there is a unique $y\in \mathbb{X}$ such that
\begin{equation*}
\xi + \varrho x -  y \in \partial (h + \tfrac{\varrho}{2}\|\cdot\|^2)(x +  y)
\quad \mbox{and} \quad 
\|y\| \le \sqrt{\varepsilon}.
\end{equation*}

Now, let $z \coloneqq x +  y$ and $\zeta \coloneqq  \xi +  (1 + \varrho) (x - z)$.
From the last display, we have
\begin{equation*}
\|z - x\| = \|y\| \le \sqrt{\varepsilon}
\quad \mbox{and} \quad \zeta + \varrho z = \xi + \varrho x -  y
\in \partial (h + \tfrac{\varrho}{2}\|\cdot\|^2)(z),
\end{equation*}	
which, together with \eqref{eq-wc-subd}, implies \eqref{eq-x-z-zeta}.

Next, we turn to \eqref{eq-f-x-z}.
Since $\zeta \in \partial h(z)$, we have
\begin{align*}
h(x) - h(z) & \ge \langle \zeta, x - z\rangle
- \tfrac{\varrho}{2} \|x - z\|^2
= \langle \xi, x - z\rangle + (1 + \tfrac{\varrho}{2}) \|x - z\|^2
 \ge -\|\xi\| \|x - z\|.
\end{align*}
On the other hand, from $\xi \in \partial_{\varepsilon}^{\varrho}h(x)$, it follows that
\begin{align*}
h(z) - h(x) & \ge \langle \xi, z-x\rangle 
- \tfrac{\varrho}{2}\|z-x\|^2 - \varepsilon
\ge - \|\xi\| \|z-x\|
- \tfrac{\varrho}{2}\|z-x\|^2 - \varepsilon.
\end{align*}
Combining the last two displays with the first relation in \eqref{eq-x-z-zeta}, we obtain \eqref{eq-f-x-z}.
\end{proof}

The following result is a slight extension of \cite[Lemma 4.2]{26XuPongZhang}, where we replace the $r_1 \in (0, 1)$ therein by $r_1 \in [0, 1)$. One can check that the proof of \cite[Lemma 4.2]{26XuPongZhang} remains valid for the case $r_1 = 0$. For brevity, we do not reproduce the proof here.

\begin{lemma}[{\cite[Lemma 4.2]{26XuPongZhang}}]
\label{lm-rate}
Let $r_1 \in [0, 1)$, $r_2 > r_1$, $p\ge 1$, $b_0 > 0$, $b_1 > 0$ and $N_1 \in \mathbb{N}_0$.
Let $\{\Delta_{k}\}$ be a nonnegative sequence such that
\begin{equation*}
\Delta_{k+1} + b_0 (1+k)^{-r_1} \Delta_{k+1}^{p} 
\le \Delta_{k} + b_1 (1+k)^{-r_2} \quad \forall k \ge N_1.
\end{equation*}
Then there exist $b_2> 0$ and $N_2\in \mathbb{N}_0$ such that
\begin{equation*}
\Delta_{k} \le b_2 (1+k)^{-s} \quad \forall k\ge N_2,
\end{equation*}
where
\begin{equation*}
s \coloneqq \begin{cases}
r_2 - r_1 & \mbox{if}\ p = 1,
\\ \min\left\{ \frac{r_2 - r_1}{p}, \frac{1-r_1}{p - 1} \right\}
& \mbox{if}\ p > 1.
\end{cases}
\end{equation*}
\end{lemma}

\section{An $\varepsilon$-subdifferential algorithmic framework for weakly convex functions}
\label{sec_applications}
In this section, we first present our abstract algorithmic framework and its global convergence result under KL property: the convergence proof will be presented later in Section~\ref{sec4}. We then describe how our framework can be readily applied to analyzing global sequential convergence and convergence rate of some widely studied algorithms whose global sequential convergence was either previously unknown or required a more complicated potential function in the analysis.

We now present our algorithmic framework below.
Let $\psi:\mathbb{X}\to (-\infty, \infty]$ be a proper, closed, 
and $\varrho$-weakly convex function with constant $\varrho\ge 0$, $\{\alpha_{k}\}$ be a sequence of positive numbers and $\{\beta_{k}\}$, $\{\nu_k\}$ and $\{\gamma_{k}\}$ be sequences of nonnegative numbers. 
Let $\{x^{k}\} \subseteq \mathbb{X}$ satisfy the following conditions:
\begin{enumerate}[{\rm (C1)}]
\item \label{C1}
There exist a constant $a_1>0$ and a sequence $\{\hat{\psi}_k\}\subset \mathbb{R}$ such that
for every $k\ge 0$,
\begin{align}
\label{eq-f-beta-des}
&\frac{a_{1}}{\alpha_k}  \| x^{k+1} - x^{k} \|^2 \le \hat{\psi}_k - \hat{\psi}_{k+1},
\\ & |\psi(x^{k}) - \hat{\psi}_k| \le \beta_{k}.
\label{eq-f-psi}
\end{align}

\item \label{C2} 
There exist constants $a_2 \ge 0$, $a_{3} \ge 0$ and $M_{\psi} > 0$ such that, for each $k\ge 0$, one can find $\xi^{k+1} \in \partial_{\widetilde{\nu}_k}^{\varrho} \psi(x^{k+1})$ satisfying
\begin{equation}\label{eq-C2}
\|\xi^{k+1}\| \le \frac{a_{3}}{\alpha_{k}} \|x^{k+1} - x^{k}\| + \gamma_{k}
\quad \mbox{and} \quad \|\xi^{k+1}\| \le M_{\psi},
\end{equation}
where
\begin{equation}\label{eq-nu-tilde}
\widetilde{\nu}_k \coloneqq a_2(\hat{\psi}_k - \hat{\psi}_{k+1})
+ \nu_k.
\end{equation}

\item \label{C3} 
The sequence $\{x^{k}\}$ is bounded.
Moreover, for every convergent subsequence $\{x^{k_j}\}$ with $\lim_{j\to\infty} x^{k_j} = \bar{x}$ for some $\bar x$, it holds that
\begin{equation}\label{eq-continuity}
\psi(\bar{x})  \in \mathbb{R} \ \ \ {\rm and}\ \ \ 
\lim_{j\to\infty}\psi(x^{k_j}) = \psi(\bar{x}).
\end{equation}
\end{enumerate} 

\begin{remark}[Discussions on the conditions (C1)-(C3)] 
Our conditions (C1)-(C3) extend the standard conditions (H1)-(H3) in the celebrated papers \cite{14BolteSabachTeboulle,13AttouchBolteSvaiter,10AttouchBolteRedontSoubeyran}, which are commonly used in KL-based convergence analysis. Indeed, if there exist $0<\underline{\alpha} < \overline{\alpha}$ such that $\underline{\alpha} \le \alpha_k \le \overline{\alpha}$ for all $k$, $\beta_k \equiv 0$, $a_2=0$, $\nu_k \equiv 0$ and $\gamma_k \equiv 0$, then our conditions (C1)-(C2) reduce to (H1)-(H2); while (C3) and (H3) are essentially the same. We next further comment on (C1) and (C2):
\begin{itemize}
\item Condition (C1) is an approximate descent condition. It requires a surrogate sequence ${\widehat{\psi}_k}$ for the function values ${\psi(x^k)}$ to satisfy a controlled descent inequality involving the successive changes $\|x^{k+1}-x^k\|$, where the amount of the descent is captured by the dynamic parameter $\alpha_k$. Unlike the classical setting, the parameter $\alpha_k$ is allowed to vary with the iteration and need not be bounded away from zero.
\item Condition (C2) is a generalized relative-error condition that permits iteration-dependent models and controls an appropriate approximate stationarity measure for nonsmooth, weakly convex functions. Our use of $\varrho$-weakly-convex-$\varepsilon$-subdifferential of $\psi$ follows the recent works \cite{24AAS,Robinson99,ASSS23}, which focused on metric subregularity and local linear convergence of their methods.
\end{itemize}
Together, (C1)-(C2) quantify the effects of model variation, inexactness, and other algorithm-dependent approximations.
With the flexibility brought by the auxiliary sequence $\{\hat\psi_k\}$ and the use of $\varrho$-weakly-convex-$\varepsilon$-subdifferential, our framework differs in a fundamental way from existing extensions of KL-based analysis; see, e.g., \cite{YKZH25,24QMLM,23QianPan,23LMQ,23HuLiu,19Ochs,21Sun,15FGP,26Tao}. 
\end{remark}


\begin{theorem}[Global convergence and convergence rate under KL property]\label{th-main}
Let $\psi:\mathbb{X}\to (-\infty, \infty]$ be a proper, closed, 
and $\varrho$-weakly convex function with $\varrho\ge 0$.
Let $\{\alpha_k\}$ be a positive sequence and $\{\beta_k\}$, $\{\gamma_k\}$ and $\{\nu_k\}$ be nonnegative sequences satisfying either {\rm (A\ref{(A)})} or {\rm (A\ref{(B)})} below.
\begin{enumerate}[{\rm ({A}1)}]
    \item \label{(A)} $\alpha_{k} = \Theta(1), \beta_k = \mathcal{O}(q_{\beta}^{k}), 
\gamma_k = \mathcal{O}(q_{\gamma}^{k})$, $\nu_k = \mathcal{O}(q_{\nu}^{k})$ 
for some $q_{\beta}, q_{\gamma}, q_{\nu} \in (0, 1)$;
\item \label{(B)} $\alpha_{k} = \Theta(k^{-r_{\alpha}})$, 
$\beta_k = \mathcal{O}(k^{-r_{\beta}})$,
$\gamma_k = \mathcal{O}(k^{-r_{\gamma}})$,
and $\nu_k = \mathcal{O}(k^{-r_{\nu}})$, with constants $r_{\alpha} \in [0, 1)$ and $r_{\beta}, r_{\gamma}, r_{\nu}>0$. Moreover, it holds that
$r_{\alpha} + r_{\widetilde{\tau}} - 1 > 0$ for some $\widetilde{\tau}\ge 2$, where 
$$r_{\tau}\coloneqq \min\left\{r_{\gamma}, \tfrac{\tau - 1}{\tau} r_{\beta}, \tfrac{1}{2} r_{\nu} \right\}
\quad \forall \tau \ge 2.$$ 
\end{enumerate}
Let $\{x^{k}\}$ satisfy {\rm (C\ref{C1})}, {\rm (C\ref{C2})} and {\rm (C\ref{C3})}.
Suppose that $\psi$ is a KL function. 
Then the following statements hold.
\begin{enumerate}[{\rm (i)}]
    \item $\{x^{k}\}$ converges to a critical point $x^*$ of $\psi$, 
and 
$\sum_{k=0}^{\infty}\|x^{k+1} - x^{k}\| < + \infty$.

\item If $\psi$ is a KL function with exponent $\theta \in (0, \frac{1}{2}]$ and condition {\rm (A\ref{(A)})} holds, then there exists $q \in (0, 1)$ such that
\begin{equation}\label{rate-main-R}
	\| x^{k}  - x^*\| = \mathcal{O}(q^{k}).
\end{equation}

\item If $\psi$ is a KL function with exponent $\theta \in (0, \frac{1}{2}]$
and {\rm (A\ref{(B)})} holds with $\widetilde{\tau}=2$, then
\begin{equation}\label{rate-main1}
	\| x^{k}  - x^*\| = \mathcal{O}(k^{- \widetilde{r}})
	\quad \text{with} 
	\quad \widetilde{r} \coloneqq  r_{\alpha} + \min\left\{r_{\gamma}, \tfrac{1}{2} r_{\beta}, \tfrac{1}{2} r_{\nu} \right\} - 1.
\end{equation}

\item If $\psi$ is a KL function with exponent $\theta \in (\frac{1}{2}, 1)$ and {\rm (A\ref{(B)})} holds, then for any $\tau\in [\widetilde{\tau}, \infty)$,
\begin{equation}\label{rate-main2}
	\| x^{k}  - x^*\| = \mathcal{O}(k^{-\widetilde{r}_{\tau}}),
\end{equation}
where
\begin{equation*}
	\widetilde{r}_{\tau} \coloneqq \min\left\{r_{\alpha} + r_{\tau} - 1, \frac{(1-\theta_{\tau})r_{\tau}}{\theta_{\tau}}, \frac{(1-\theta_{\tau})(1- r_{\alpha})}{2\theta_{\tau} - 1} \right\}
    \quad \text{and} \quad \theta_{\tau} \coloneqq \max\left\{\theta, \frac{\tau - 1}{\tau}\right\}.
\end{equation*}
\end{enumerate}
\end{theorem}

The proof of Theorem~\ref{th-main} will be presented in Section~\ref{sec4} below. For the remainder of this section, we illustrate the usefulness of our framework by showing that it covers several classical optimization methods under suitably chosen parameters $\{\alpha_{k}\}$, $\{\beta_{k}\}$, $\{\nu_k\}$ and $\{\gamma_{k}\}$; this allows us to establish their global sequential convergence and derive their convergence rates as corollaries of Theorem~\ref{th-main}. Moreover, we will see that the potential function $\psi$ can be chosen naturally as either the original objective function or the classical exact penalty function (in the presence of explicit equality / inequality constraints).



\subsection{Application I: Bundle and proximal model-based methods}\label{sec3.1}
Bundle methods were originally developed for nonsmooth convex optimization. The basic idea is to construct a tractable cutting-plane model from accumulated first-order information every iteration and stabilize the model subproblem by a proximal term; see, e.g., \cite{06Kiwiel,Kelley60}. These methods have been extended to handle some nonsmooth nonconvex optimization problems through various proximal and bundle-type constructions; see, e.g., \cite{04FGG,09HS,10HS,16HSS} and references therein.

Recently, Atenas et al.~\cite{ASSS23} developed a unified convergence analysis for a broad class of descent schemes for weakly convex optimization, encompassing, in particular, proximal model-based and bundle methods.
Here, we show that the class of bundle-type methods described in \cite[Section~5.2]{ASSS23} fits in our algorithmic framework described in (C1), (C2) and (C3). To describe the method, we first recall the following definition, which is taken from \cite[Definition~5.1]{ASSS23}. 

\begin{definition}\label{def-1QA-plus}
    A convex function $\varphi_x:\mathbb{R}^n \to \mathbb{R}$ is said to model a proper function 
    $\psi: \mathbb{R}^n \to (-\infty, \infty]$ with {\em one-sided} quadratic accuracy (1QA) at $x\in \dom \psi$ with parameter $\ell > 0$ 
    if 
    \begin{equation}\label{1QA}
        \varphi_x(y) \le \psi(y) - \psi(x) + \frac{\ell}{2}\|y - x\|^2 
        \quad \forall y \in \mathbb{R}^n.
    \end{equation}
    We call $\varphi_x$ a 1QA model of $\psi$ at $x$ with parameter $\ell$.
\end{definition}

In this subsection, we consider the bundle-type algorithmic framework presented in Algorithm~\ref{alg-prox-model}, which is based on 1QA models and is taken from \cite[Eq.~(5.2)]{ASSS23}.\footnote{Note that the boundedness condition on $\{\alpha_k\}$, though not explicitly included in \cite[Eq.~(5.2)]{ASSS23}, was assumed in the main convergence result \cite[Theorem~5.7]{ASSS23}. For succinctness of presentation, we include this condition as an assumption on algorithmic parameters.} We refer the readers to \cite[Section~5.2]{ASSS23} for discussions of the bundle-type methods covered by this algorithmic framework and efficient routines for finding $x^{k+1}$. Notice that, by induction, one can show that $x^k\in {\rm dom}\,\psi$ for all $k\ge 0$.

\begin{algorithm}[h]
\caption{A descent proximal model-based algorithmic scheme}
\label{alg-prox-model}
\begin{algorithmic}[0]
\setlength{\columnwidth}{\linewidth}
\State {\bf Require.} $x^0\in \dom\psi$, $\ell>0$, $c\in (0,1)$,
and $\{\alpha_k\} \subset [\underline \alpha, \bar \alpha]$ for some $\bar{\alpha}\ge \underline{\alpha} >0$.
\For{$k=0,1,\ldots$} 
\State Construct a 1QA model $\varphi_{x^k}$ of $\psi$ with parameter $\ell$ at $x^k$ such that 
\begin{align}
  \psi(\widetilde x) - \psi(x^k) \le c \, \varphi_{x^k}(\widetilde x), \label{model-base-2}
\end{align}
\State where 
\begin{align}\label{model-base-1}
    \widetilde x \coloneqq \argmin_{y\in \mathbb{R}^n} \bigg\{ \varphi_{x^k}(y) + \frac{1}{2\alpha_k}\| y - x^{k}\|^2\bigg\}.
\end{align}
\State Set $x^{k+1} = \widetilde x$.
\EndFor
\end{algorithmic}
\end{algorithm}

Linear convergence of Algorithm~\ref{alg-prox-model} was established in \cite[Theorem~5.7]{ASSS23} under metric subregularity and an isocost condition for a weakly convex function $\psi$, assuming in addition that $\inf \psi > -\infty$ and $\{x^k\}$ is bounded. 
Here, we apply our convergence analysis framework to derive the convergence of $\{x^k\}$ when $\psi$ is a level-bounded weakly convex KL function, and discuss the explicit local convergence rate according to the KL exponent. Our result can be seen as an extension of \cite[Theorem~5.7]{ASSS23} because 
the assumptions there imply that $\psi$ satisfies the KL property with exponent $1/2$ (see, e.g., \cite[Proposition 3.8]{21DrusvyatskiyIoffeLewis}).

We first present some basic properties of $\{x^k\}$ generated by Algorithm~\ref{alg-prox-model} in the next lemma.
Note that the property of the $\varrho$-w.c.-$\varepsilon$-subdifferential analogous to that in Lemma~\ref{lm-prox-bundle}(iii) below was previously established in \cite[Proposition~5.4]{ASSS23}.

\begin{lemma}\label{lm-prox-bundle}
    Let $\psi: \mathbb{R}^n \to (-\infty, \infty]$ be a proper closed level-bounded $\varrho$-weakly convex function with constant $\varrho\ge 0$. 
    Let $\{x^k\}$ be generated by Algorithm~\ref{alg-prox-model} and we denote the 1QA model with parameter $\ell > 0$ associated with each $x^k$ by $\varphi_{x^k}$. 
    Then, the following statements hold.
    \begin{enumerate}[{\rm (i)}]
       \item It holds that
        \begin{align}\label{des-model-base}
           \frac{c}{\alpha_k}\|x^{k+1} - x^k\|^2 \le \psi(x^{k}) - \psi(x^{k+1})
           \quad \forall k\ge 0.
        \end{align}

 \item For each $k\ge 0$, it holds that $(x^k-x^{k+1})/\alpha_k \in \partial \varphi_{x^k}(x^{k+1})$.
        \item For each $k\ge 0$, it holds that $(x^k-x^{k+1})/\alpha_k \in \partial_{\widetilde{\nu}_k}^{\widetilde{\varrho}} \psi (x^{k+1})$, 
        where 
        \begin{equation}
           \widetilde{\varrho} \coloneqq \max\{2 \ell, \varrho\}
           \quad \text{and} \quad
           \widetilde{\nu}_k \coloneqq (c^{-1}(1+ \ell \bar \alpha) - 1)(\psi(x^{k}) - \psi(x^{k+1})) \ge 0.
              \label{nu-varrho}
        \end{equation}
        \item The sequences $\{x^k\}$ and $\{(x^k-x^{k+1})/\alpha_k\}$ are both bounded. 
        Moreover, for any convergent subsequence $\{x^{k_j}\}$ with $\lim_{j\to\infty} x^{k_j} = \bar x$ for some $\bar x$, it holds that
        $$\lim_{j\to\infty} \psi(x^{k_j}) = \psi(\bar x) \in \mathbb{R}.$$
    \end{enumerate}
\end{lemma}
\begin{proof}
First, we consider item (i). 
Since the objective of \eqref{model-base-1} is $\frac{1}{\alpha_k}$-strongly convex, we have
\begin{align}\label{varphi-dx}
\varphi_{x^k}(x^{k+1}) + \frac{1}{2\alpha_k}\|x^{k+1} - x^k\|^2 
&\le \varphi_{x^{k}}(x^k) - \frac{1}{2\alpha_k} \|x^{k+1} - x^k\|^2\notag\\
&\le - \frac{1}{2\alpha_k} \|x^{k+1} - x^k\|^2,
\end{align}
where the second inequality holds by \eqref{1QA}.
This, together with \eqref{model-base-2}, implies \eqref{des-model-base}.

Item (ii) follows directly from the first-order optimality condition of \eqref{model-base-1}.
We now turn to the proof of item (iii). 
   As $c<1$, we have $c < 1+\ell \bar{\alpha}$, and so, \eqref{des-model-base} implies that $\widetilde{\nu}_k\ge 0.$ 
Moreover, from \eqref{1QA}, for any $x\in \mathbb{R}^n$,
    \begin{align*}
        & \psi(x)  \ge \psi(x^{k}) + \varphi_{x^k}( x) - \frac{\ell}{2} \| x - x^{k} \|^2
        \\ & \overset{\text{\rm{(a)}}}{\ge} \psi(x^{k}) + \varphi_{x^k}( x^{k+1}) + \langle \xi^{k+1}, x - x^{k+1}\rangle 
         - \ell \| x - x^{k+1} \|^2 - \ell \| x^{k+1} - x^k\|^2
         \\ & \overset{\text{\rm{(b)}}}{\ge} \psi(x^{k})  + \langle \xi^{k+1}, x - x^{k+1}\rangle 
         - \ell \| x - x^{k+1} \|^2 + (1+ \ell \alpha_k) \varphi_{x^k}( x^{k+1})
         \\ & \overset{\text{\rm{(c)}}}{\ge} \psi(x^{k})  + \langle \xi^{k+1}, x - x^{k+1}\rangle 
         - \ell \| x - x^{k+1} \|^2 + (1+ \ell \alpha_k)c^{-1} (\psi(x^{k+1}) - \psi(x^{k}) )
        \\ & \overset{\text{\rm{(d)}}}{\ge} \psi(x^{k+1}) + \langle \xi^{k+1}, x - x^{k+1}\rangle - \frac{\widetilde{\varrho}}{2}\| x - x^{k+1} \|^2 
        - (c^{-1}(1+ \ell \bar \alpha) - 1) (\psi(x^{k}) - \psi(x^{k+1}) ),
    \end{align*}
    where (a) follows from the inequality $\|u+v\|^2\le 2\|u\|^2+2\|v\|^2$ for $u,v\in \mathbb{R}^n$ 
    and the convexity of $\varphi_{x^k}$ with $\xi^{k+1}\coloneqq (x^k-x^{k+1})/\alpha_k \in \partial \varphi_{x^k}(x^{k+1})$ (see item (ii)), 
    (b) follows from \eqref{varphi-dx}, 
    (c) follows from \eqref{model-base-2},
    and (d) holds because $\ell\le \widetilde{\varrho}/2$, $\alpha_k\le \bar{\alpha}$ 
    and $\psi(x^{k}) - \psi(x^{k+1})\ge 0$ (see \eqref{des-model-base}).
    Item (iii) follows immediately from the above display.
    
    Next, we consider item (iv). 
    By assumption, the level set $\Omega_0 \coloneqq \{ x : \psi(x) \le \psi(x^0)\}$ is bounded. 
    Then in view of item (i), we see that there exists $M>0$ such that $\|x^k\|\le M$ for all $k\ge 0$.
    Hence, for every $k\ge 0$, we have
    \begin{equation}\label{boundedness}
    \alpha_k^{-1} \| x^{k+1} - x^k\| 
    \le \underline \alpha^{-1}(\| x^{k+1}\| + \|x^k\|) 
    \le 2 M / \underline \alpha.
    \end{equation}
    
    Let $\{x^{k_j}\}$ be a subsequence satisfying $\lim_{j\to\infty} x^{k_j} = \bar x$. 
    Since $\psi$ is closed, we have $\liminf_{j\to \infty} \psi(x^{k_j}) \ge \psi(\bar x)$.
    By the monotonicity of $\{\psi(x^k)\}$ (see item (i)), the properness of $\psi$ and the fact that $x^0 \in \dom \psi$, 
    we have $\psi(\bar x)\in \mathbb{R}$. 
    Next, from item (iii) and the definition of $\varrho$-weakly-convex-$\varepsilon$-subdifferential (see \eqref{eq-eps}), we have
    \[
    \psi(\bar x)\ge \psi(x^{k_j}) + \langle (x^{k_j-1} - x^{k_j})/\alpha_{k_j-1},\bar x - x^{k_j}\rangle - \frac{\widetilde\varrho}2\|\bar x - x^{k_j}\|^2 - \widetilde{\nu}_{k_j-1}.
    \]
    Passing to the limit as $j\to \infty$ on both sides of the above display, and invoking \eqref{boundedness} and noting that $\widetilde{\nu}_{k}\to 0$ (thanks to the level-boundedness of $\psi$, item (i), and the definition of $\widetilde{\nu}_k$ in \eqref{nu-varrho}), 
    we obtain $\limsup_{j\to \infty} \psi(x^{k_j}) \le \psi(\bar x)$. 
    Therefore, $\lim_{j\to \infty} \psi(x^{k_j}) = \psi(\bar x) \in \mathbb{R}$. 
\end{proof}

We are now ready to establish the convergence of Algorithm~\ref{alg-prox-model} by showing that the $\{x^k\}$ generated satisfies (C1), (C2) and (C3) with suitable parameters, and invoking Theorem~\ref{th-main}. To the best of our knowledge, this is the first global sequential convergence and convergence rate results based on the KL property for this class of methods.

\begin{theorem}[Global convergence of Algorithm~\ref{alg-prox-model}]
 Let $\psi: \mathbb{R}^n \to (-\infty, \infty]$ be a proper closed $\varrho$-weakly convex function 
 with constant $\varrho\ge 0$.
Suppose further that $\psi$ is a level-bounded KL function.
Let $\{x^k\}$ be generated by Algorithm~\ref{alg-prox-model}.
Then the following statements hold.
\begin{enumerate}[{\rm (i)}]
    \item $\{x^{k}\}$ converges to a critical point $x^*$ of $\psi$ and  
    $\sum_{k=0}^{\infty}\|x^{k+1} - x^{k}\| < + \infty$.
    \item If $\psi$ is a KL function with exponent $\theta \in (0, \frac{1}{2}]$, then there exists $q \in (0, 1)$ such that
\[
\| x^{k}  - x^*\| = \mathcal{O}(q^{k}).
\]

\item If $\psi$ is a KL function with exponent $\theta \in (\frac{1}{2}, 1)$, then
$\| x^{k} - x^*\| = \mathcal{O}(k^{-\frac{1 - \theta}{2 \theta - 1}})$.
\end{enumerate}
\end{theorem}
\begin{proof}
Since $\widetilde \varrho\ge \varrho$, $\psi$ is also $\widetilde \varrho$-weakly convex, 
where $\widetilde{\varrho}$ is given in \eqref{nu-varrho}.
Let $\hat{\psi}_k \coloneqq \psi(x^{k})$, 
$\beta_k \coloneqq 0$, $\gamma_k \coloneqq 0$, 
$\nu_k \coloneqq 0$, and $\xi^{k+1} \coloneqq (x^k - x^{k+1})/\alpha_k$ for all $k\ge 0$.
Then, the sequence $\{x^k\}$ satisfies {\rm (C\ref{C1})} by Lemma~\ref{lm-prox-bundle}(i), 
satisfies {\rm (C\ref{C2})} by Lemma~\ref{lm-prox-bundle}(iii)-(iv), 
and satisfies {\rm (C\ref{C3})} by Lemma~\ref{lm-prox-bundle}(iv).

Clearly, condition {\rm (A\ref{(A)})} in Theorem~\ref{th-main} holds.
Thus, items (i) and (ii) follow from Theorem~\ref{th-main}(i) and (ii), respectively.

Finally, we consider item (iii). 
Suppose that $\psi$ is a KL function with an exponent $\theta \in (\frac{1}{2}, 1)$.
Note that $\alpha_{k} = \Theta(k^{-r_{\alpha}})$ with $r_{\alpha}= 0$, 
and that $\beta_k = \mathcal{O}(k^{-r})$, $\gamma_k = \mathcal{O}(k^{-r})$, 
and $\nu_k =\mathcal{O}(k^{-r})$, for any $r\in (0, \infty)$.
For $r > 2$, it is straightforward to verify that {\rm (A\ref{(B)})} in Theorem~\ref{th-main} holds with $\widetilde{\tau} = 2$.
Therefore, the conclusion of Theorem~\ref{th-main}(iv) holds. 
By choosing $r$ sufficiently large and setting $\tau = 2$ in \eqref{rate-main2}, 
we then obtain item (iii).
\end{proof}

\subsection{Application II: Quadratic (smooth) penalty method}\label{sec3.2}
In this section, we consider the following equality constrained minimization problem:
\begin{align}\label{opt-NLP-eq}
\notag		& \min_{x\in \mathbb{R}^{n}} \ \ f(x) + \phi(x) + \delta_{\mathcal{C}}(x)
\\ & \ \ \mbox{s.t.} \quad G(x) = 0,
\end{align}
where $\delta_{\cal C}$ is the indicator function of the set ${\cal C}$, and  $\mathcal{C}$, $f$, $\phi$ and $G$ satisfy the following standard assumption; see, e.g., the recent work \cite{26LMX} and references therein for similar assumptions. 
\begin{assumption}\label{ass-ec}
In \eqref{opt-NLP-eq},
\begin{enumerate}
[{\rm (i)}]
\item the set $\mathcal{C} \subseteq \mathbb{R}^n$ is closed and convex, 
and the feasible set $\mathcal{F}\coloneqq \{ x \in \mathcal{C} :  G(x) = 0 \}\neq \emptyset$;

\item the function $\phi:\mathbb{R}^{n} \to \mathbb{R}$ is convex and 
$M_{\phi}$-Lipschitz continuous on some open set $\mathcal U$ containing $\mathcal{C}$ with 
constant $M_{\phi} > 0$; 

\item the function $f:\mathbb{R}^{n}\to\mathbb{R}$ is continuously differentiable, 
and there exist constants $M_f>0$ and $L_f>0$ such that
\begin{equation}\label{assume_Lipf}
    \|\nabla f(x)\| \le M_f \quad \forall x\in \mathcal{C}
    \quad \text{and} \quad
    \| \nabla f(x) - \nabla f(y)\| \le L_f \| x - y \|
    \quad \forall x,y\in \mathcal{C};
\end{equation}

\item  the mapping $G:\mathbb{R}^{n} \to \mathbb{R}^{m}$ is 
continuously differentiable, and there exist constants $M_G>0$ and $L_G>0$ such that
\begin{equation*}
        \|D G(x)\| \le M_G \quad  \forall x\in \mathcal{C}
    \quad \text{and} \quad
    \| D G(x) - D G(y)\| \le L_G \| x - y \|
    \quad \forall x,y\in \mathcal{C};
\end{equation*}
\item (constraint qualification) there exists $\widetilde{\kappa}>0$ such that
    \begin{equation}\label{eq-eb-ec}
        \dist(0, D G(x)^{*} G(x) + \partial\delta_{\cal C}(x) ) \ge \widetilde{\kappa} \| G(x) \|
        \quad \forall x \in \mathcal{C};
    \end{equation}
\item (level boundedness) there exists $\rho_0 > 0$ such that $Q_{\rho_0}$ is level-bounded,
where $Q_{\rho}:\mathbb{R}^n \to (-\infty, \infty]$ is the standard quadratic penalty function defined by
\begin{align}\label{eq-Q-rho}
& Q_{\rho}(x) \coloneqq f(x) + \phi(x) + \delta_{\mathcal C}(x) + \frac{\rho}{2} \| G(x) \|^2 \quad \forall x \in \mathbb{R}^n,
\end{align}
with $\rho>0$ being a penalty parameter.
\end{enumerate}
\end{assumption}

\begin{remark}
    Assumption~\ref{ass-ec}(v) plays a pivotal role in the design of algorithms that can generate nearly feasible solutions to problem \eqref{opt-NLP-eq}, and is commonly adopted in the literature; see, e.g., \cite{19SALC,22LMX,24AlacaogluWright,24LCLLX,26LMX}.
\end{remark}

\begin{remark}
Let \begin{equation}\label{eq-Lh-pen}
h(x) \coloneqq \frac{1}{2} \| G(x) \|^2 
\quad \text{and} \quad 
L(x) \coloneqq M_{G}^2 + L_G \|G(x)\|.
\end{equation}
 Under Assumption~\ref{ass-ec}, it can be directly verified that, for all $x,y\in \mathcal{C}$,
\begin{align}\label{des-lm-h}
     & \| \nabla h(y) - \nabla h(x) \| \le L(x) \| y - x\|,
    \\ & h(y) \le h(x) + \langle \nabla h(x), y - x\rangle + \frac{L(x)}{2}\|y-x\|^2.
    \label{Lips-Dh}
\end{align}
\end{remark}
Now, we present a prototypical quadratic penalty proximal gradient method for solving \eqref{opt-NLP-eq} as Algorithm~\ref{alg-qpen-pg}. The formula for updating $x^{k+1}$ in \eqref{opt-sub-ec} is obtained by applying one step of the proximal gradient method (with stepsize $\alpha_k > 0$) to the standard quadratic penalty function (see \eqref{eq-Q-rho}) for \eqref{opt-NLP-eq}. The next lemma collects some useful properties of the sequence $\{x^k\}$ generated by Algorithm~\ref{alg-qpen-pg}.
\begin{algorithm}[h]
\caption{Quadratic penalty proximal gradient method for \eqref{opt-NLP-eq}} 
\label{alg-qpen-pg}
\begin{algorithmic}[0]
\setlength{\columnwidth}{\linewidth}
\State {\bf Require.} $x^0 \in \mathcal{C}$, the parameters $\rho_0$, $L_f$, $M_{G}$, $L_G$ from Assumption~\ref{ass-ec}, and $\{\rho_{k}\}_{k\ge 1}\subset \mathbb{R}_{++}$.
\For{$k=0,1,\ldots$} 
\State Let $\alpha_{k} = 1/(L_f + \rho_k L_k)$, where $L_k\coloneqq M_{G}^2 + L_G \|G(x^{k})\| $.
\State Compute
\begin{equation}
    \label{opt-sub-ec}
    x^{k+1} = \argmin_{x \in \mathbb{R}^n}
    \bigg\{
    \phi(x) + \delta_{\mathcal{C}}(x) 
    + \langle d^k, x \rangle 
    + \frac{1}{2\alpha_k}\| x - x^{k} \|^2\bigg\},
\end{equation}
\State where $d^{k} \coloneqq  \nabla f(x^{k}) + \rho_{k} DG(x^{k})^{*} G(x^{k})$.
\EndFor
\end{algorithmic}
\end{algorithm}

\begin{lemma}\label{lm-bs-ec}
Consider \eqref{opt-NLP-eq} and suppose that Assumption~\ref{ass-ec} holds. 
Let $Q_{\rho}$ be given in \eqref{eq-Q-rho}.
Then Algorithm~\ref{alg-qpen-pg} is well-defined and generates a sequence $\{x^{k}\} \subseteq \mathcal{C}$.
Moreover, the following statements hold for each $k\ge 0$.
\begin{enumerate}[{\rm (i)}]
\item There exist $\zeta^{k+1}_{\phi} \in \partial \phi(x^{k+1})$ and $\zeta^{k+1}_{\mathcal{C}} \in \partial\delta_{\cal C}(x^{k+1})$ such that
\begin{align}\label{eq-crit-ec}
0 & = \nabla f(x^{k}) + \rho_{k} DG(x^{k})^{*} G(x^k) + \frac{1}{ \alpha_k}( x^{k+1} - x^{k}) + \zeta^{k+1}_{\phi} + \zeta^{k+1}_{\mathcal{C}}.
\end{align}

\item It holds that 
\begin{align}\label{eq-Q-des}
    & Q_{\rho_{k+1}}(x^{k+1}) \le Q_{\rho_k}(x^{k}) - \frac{1}{2\alpha_k} \| x^{k+1} - x^k \|^2 
    + \frac{\rho_{k+1} - \rho_k}{2} \| G(x^{k+1})\|^2.
\end{align}

\item It holds that
\begin{align}\label{eq-xi-ub-ec1}
& \|\xi^{k+1}\| \le \frac{2}{\alpha_k} \|x^{k+1} - x^{k}\| \le 2(M_f + M_{\phi} + M_G \rho_k \| G(x^{k})\|),
\end{align}
where $\xi^{k+1} \coloneqq \nabla f(x^{k+1}) + \rho_{k} DG(x^{k+1})^{*} G(x^{k+1}) + \zeta^{k+1}_{\phi} + \zeta^{k+1}_{\mathcal{C}}$, with $\zeta^{k+1}_\phi$ and $\zeta^{k+1}_{\cal C}$ given in \eqref{eq-crit-ec}.
\item  It holds that
\begin{align}
h(x^{k+1}) +  \frac{\rho_k \alpha_k}{32} \|\widetilde{\xi}^{k+1}\|^2
 \le  h(x^k) + \frac{17(M_f + M_{\phi})^2\alpha_k}{16\rho_k},
 \label{eq-h-des-1}
\end{align}
where $\widetilde{\xi}^{k+1} \coloneqq  DG(x^{k+1})^{*} G(x^{k+1}) + \rho_k^{-1} \zeta^{k+1}_{\mathcal{C}}$, $h$ is defined in \eqref{eq-Lh-pen}, and $\zeta^{k+1}_{\cal C}$ is given in \eqref{eq-crit-ec}.
\end{enumerate}
\end{lemma}
\begin{proof}
For each $k=0,1,\ldots$, since the objective function in \eqref{opt-sub-ec} is proper, closed, and $\frac{1}{\alpha_k}$-strongly convex, 
it has a unique solution $x^{k+1} \in \mathcal{C}$. 
Hence, Algorithm~\ref{alg-qpen-pg} is well-defined, $\{x^{k}\} \subseteq \mathcal{C}$, 
and item (i) follows directly from the first-order optimality condition of the subproblem \eqref{opt-sub-ec}. 

Now, we consider item (ii).
Let $k\ge 0$.
From the $\frac{1}{\alpha_k}$-strong convexity of \eqref{opt-sub-ec}, it follows that
\begin{equation}\label{eq-sub-des-ec}
    \phi(x^{k+1})  + \langle d^k, x^{k+1} \rangle 
    + \frac{1}{2\alpha_k}\| x^{k+1} - x^{k} \|^2
    \le \phi(x^k)   + \langle d^k, x^k \rangle 
    - \frac{1}{2\alpha_k}\| x^{k+1} - x^{k} \|^2.
\end{equation}
Using the $L_f$-smoothness of $f$ on ${\cal C}$, \eqref{Lips-Dh} and the fact $L_f + \rho_k L_k = 1/\alpha_k$, we have
\begin{equation*}
    f(x^{k+1}) + \rho_k h(x^{k+1}) \le f(x^{k}) + \rho_k h(x^{k}) + \langle d^{k}, x^{k+1} - x^{k}\rangle + \frac{1}{2\alpha_k}\|x^{k+1} - x^{k}\|^2.
\end{equation*}
Summing the last two displays and rearranging terms, we obtain \eqref{eq-Q-des}.

Next, we turn to item (iii). From \eqref{eq-crit-ec} and the definition of $\xi^{k+1}$, 
it follows that
\begin{align*}
0 & = \nabla f(x^{k}) - \nabla f(x^{k+1}) + \rho_{k} (\nabla h(x^{k}) - \nabla h(x^{k+1}) ) + \frac{1}{ \alpha_k}( x^{k+1} - x^{k}) + \xi^{k+1},
\end{align*}
Using this, \eqref{assume_Lipf} and \eqref{des-lm-h}, we have
\begin{equation}\label{eq-xi-up-ec}
\| \xi^{k+1}\| \le ( L_f + \rho_k L_k + \tfrac{1}{\alpha_k}) \| x^{k+1} - x^{k} \| 
= \tfrac{2}{\alpha_k} \| x^{k+1} - x^{k} \|.
\end{equation}
On the other hand, from \eqref{eq-sub-des-ec}, it follows that
\begin{align*}
\frac{1}{\alpha_k}\| x^{k+1} - x^{k} \|^2
& \le \phi(x^k) -  \phi(x^{k+1})  + \langle d^k, x^k - x^{k+1}\rangle
\\ & \le (M_{\phi} + M_{f} + M_{G} \rho_k \| G(x^k) \|)  \| x^{k+1} - x^k\|,
\end{align*}
where the second inequality holds thanks to Assumption~\ref{ass-ec}(ii), (iii) and (iv) upon recalling that $d^k = \nabla f(x^k) + \rho_k DG(x^k)^* G(x^k)$.
The above display together with \eqref{eq-xi-up-ec} implies \eqref{eq-xi-ub-ec1}.

Finally, we consider item (iv). 
From \eqref{eq-sub-des-ec} and the definition of $d^k$, it follows that
\begin{align}
\notag \rho_k \langle  \nabla h(x^k), x^{k+1} - x^k \rangle  
& \le \phi(x^k)  - \phi(x^{k+1}) - \langle \nabla f(x^k), x^{k+1} - x^k \rangle
- \tfrac{1}{\alpha_k}\| x^{k+1} - x^{k} \|^2
\\ & \le (M_f + M_\phi) \| x^{k+1} - x^k\| - \tfrac{1}{\alpha_k}\| x^{k+1} - x^{k} \|^2 \notag
\\ & \le (M_f + M_\phi)^2\alpha_k - \frac{3}{4\alpha_k}\| x^{k+1} - x^{k} \|^2, \label{eq-hdes-ec}
\end{align}
where the second inequality holds thanks to Assumption~\ref{ass-ec}(ii) and (iii), 
and the third inequality holds because $ab \le \frac{\eta a^2}{2} + \frac{b^2}{2\eta}$
with $a = M_f + M_{\phi}$, $b = \|x^{k+1} - x^k\|$ and $\eta = 2\alpha_k$.
Now, using \eqref{Lips-Dh}, we have
\begin{align}
\notag    h(x^{k+1}) & \le h(x^k) + \langle \nabla h(x^{k}), x^{k+1} - x^{k} \rangle + \frac{L_k}{2} \| x^{k+1} - x^k \|^2
    \\ & \le  h(x^k) +  \frac{(M_f + M_\phi)^2\alpha_k}{\rho_k} 
    - \left( \frac{3}{4\alpha_k\rho_k} - \frac{L_k}{2}\right) \| x^{k+1} - x^k \|^2 \notag
        \\ & \le  h(x^k) +  \frac{(M_f + M_\phi)^2\alpha_k}{\rho_k} 
    - \frac{1}{4\alpha_k\rho_k}\| x^{k+1} - x^k \|^2, \label{eq-h-des-ec}
\end{align}
where the second inequality follows from \eqref{eq-hdes-ec}, 
and the third inequality holds because
\begin{equation}\label{eq-Lh-alpha-rho}
    \alpha_k \rho_k = \rho_k/(L_f + \rho_k L_k) \le 1/ L_k.
\end{equation}

On the other hand, using \eqref{eq-crit-ec} and the definition of $\widetilde{\xi}^{k+1}$, we have 
\begin{align}
     \notag 
    \| \widetilde{\xi}^{k+1}\| & = \left\| \nabla h(x^{k+1}) - \nabla h(x^{k})
    - \frac{1}{\rho_k}(\nabla f(x^k) + \zeta_{\phi}^{k+1}) - \frac{1}{\alpha_k\rho_k} (x^{k+1} - x^k) \right\| \notag
    \\ \notag & \le \| \nabla h(x^{k+1}) - \nabla h(x^{k})\| + \frac{M_f + M_{\phi}}{\rho_k} + \frac{1}{\alpha_k\rho_k} \| x^{k+1} - x^k\|
    \\ & \le \frac{M_f + M_{\phi}}{\rho_k} + \frac{2}{\alpha_k \rho_k} \| x^{k+1} - x^k\|, \notag
\end{align}
where the first inequality follows from Assumption~\ref{ass-ec}(ii) and (iii), 
and the second inequality holds thanks to \eqref{des-lm-h} and $L_k \le 1/(\alpha_k \rho_k)$ (see \eqref{eq-Lh-alpha-rho}).
This implies 
\begin{align}
 \notag\frac{\rho_k \alpha_k}{32} \|\widetilde{\xi}^{k+1}\|^2
 & \le \frac{\rho_k \alpha_k}{32} \left( \frac{2(M_f + M_{\phi})^2}{\rho_k^2} + \frac{8}{(\alpha_k \rho_k)^2} \| x^{k+1} - x^k\|^2\right)
 \\ & =  \frac{(M_f + M_{\phi})^2\alpha_k}{16\rho_k} + \frac{1}{4\alpha_k \rho_k} \| x^{k+1} - x^k\|^2.\notag
\end{align}
Summing the last display with \eqref{eq-h-des-ec}, 
we obtain \eqref{eq-h-des-1}.
\end{proof}	

Next, we turn to establishing some relations among the $\{x^k\}$ generated by Algorithm~\ref{alg-qpen-pg}, the quadratic penalty function $Q_{\rho}$ in \eqref{eq-Q-rho}
and the $\ell_1$-based penalty function $E_{\mu}$ defined below:
\begin{equation}\label{eq-E-rho}
E_{\mu}(x) \coloneqq f(x) + \phi(x) + \delta_{\mathcal C}(x) + \mu \|G(x)\| \quad \forall x \in \mathbb{R}^{n},
\end{equation}
where $\mu>0$ is a penalty parameter. 
Notice that, under Assumption~\ref{ass-ec}, for every $\mu>0$, $E_\mu$ is $(L_f+\mu L_G)$-weakly convex,
since the function $x\mapsto \|G(x)\|+\delta_{\cal C}(x)$ is $L_G$-weakly convex (see, e.g., \cite[Lemma~4.2]{19DrusvyatskiyPaquette}), 
$f$ is $L_f$-smooth on ${\cal C}$ and $\phi+\delta_{\mathcal C}$ is convex. This function will also be used in our main convergence theorem later.
\begin{lemma}\label{lm-E-Q}
Consider \eqref{opt-NLP-eq} and suppose that Assumption~\ref{ass-ec} holds.
Let $\{\rho_k\}$ be a nondecreasing positive sequence such that $\rho_k = \Theta(k^{r})$ and $\rho_{k+1}- \rho_k = \mathcal{O}(k^{-(1-r)})$ for some $r\in (\frac{1}{2}, 1)$.
Let $\{x^{k}\}$ be generated by Algorithm~\ref{alg-qpen-pg}.
Let $Q_{\rho}$ and $E_{\mu}$ be given in \eqref{eq-Q-rho} and \eqref{eq-E-rho}, respectively. 
Then the following statements hold.
\begin{enumerate}[{\rm (i)}]
\item $\alpha_k = \Theta(k^{-r})$, and there exists a constant $\widetilde{\mu}>0$ such that 
\begin{equation}\label{eq-wt-mu}
    \rho_k \max\{\|G(x^{k})\|, \|G(x^{k+1})\|\} \le \widetilde{\mu}
\quad \forall k\ge 0.
\end{equation}

\item The sequence $\{x^k\} \subseteq \mathcal{C}$ is bounded, and it holds that
\begin{align}\label{eq-Q-des-1}
& Q_{\rho_{k+1}}(x^{k+1}) \le Q_{\rho_k}(x^{k}) - \frac{1}{2\alpha_k} \| x^{k+1} - x^k \|^2 
+ \widetilde{\beta}_k 
\quad \forall k\ge 0,
\end{align}
where $\widetilde{\beta}_{k} \coloneqq \frac{1}{2} (\rho_{k+1} - \rho_k) \| G(x^{k+1})\|^2 
= \mathcal{O}(k^{-(1 + r)})$.

\item We have
\begin{equation}\label{eq-E-Q}
    \frac{\widetilde{\mu}}{2} \|G(x^k)\| \le E_{\widetilde\mu}(x^{k}) - 
Q_{\rho_k}(x^{k}) \le \frac{\widetilde\mu^2}{2\rho_{k}} \quad \forall k \ge 0,
\end{equation}
where $\widetilde{\mu}$ is given in \eqref{eq-wt-mu}.
\item For each $k\ge 0$, there exists $\xi^{k+1} \in \partial_{\nu_k}^{\widetilde{\varrho}} E_{\widetilde{\mu}}(x^{k+1})$ such that
\begin{equation}\label{eq-C2-ec}
\|\xi^{k+1}\| \le \frac{2}{\alpha_k} \|x^{k+1} - x^{k}\| \le 2(M_f + M_{\phi} + M_G\, \widetilde{\mu}),
\end{equation}
where $\widetilde{\mu}$ is given in \eqref{eq-wt-mu}, $\widetilde{\varrho} \coloneqq  L_f + \widetilde{\mu} L_G$ and 
$\nu_k \coloneqq \widetilde{\mu}^2/\rho_{k+1}$.
\end{enumerate}
\end{lemma}
\begin{proof}
First, we consider item (i). 
Let $k\ge 0.$
Let $\widetilde{\xi}^{k+1}$ be given as in \eqref{eq-h-des-1}.
By definition, we have
\begin{equation*}
    \|\widetilde{\xi}^{k+1}\|^2 \ge \dist\left(0, DG(x^{k+1})^{*} G(x^{k+1}) 
    + \partial\delta_{\cal C}(x^{k+1})\right)^2
    \ge \widetilde{\kappa}^2\|G(x^{k+1})\|^2 
    = 2 \widetilde{\kappa}^2 h(x^{k+1}),
\end{equation*}
where the second inequality follows from \eqref{eq-eb-ec} and the equality follows from the definition of $h$ in \eqref{eq-Lh-pen}.
This, together with \eqref{eq-h-des-1}, yields
\begin{align}
\notag h(x^{k+1}) +  \frac{\widetilde{\kappa}^2}{16} \rho_k \alpha_k h(x^{k+1})
& \le  h(x^k) + \frac{17(M_f + M_{\phi})^2\alpha_k}{16\rho_k} 
\quad \forall k\ge 0.
\end{align}

Note that $\alpha_k  = 1/(L_f +\rho_k ( M_{G}^2 + L_{G} \|G(x^{k})\| )) \le 1/(\rho_k  M_{G}^2)$ for all $k\ge 0$, 
which, together with $\rho_k= \Theta(k^{r})$, implies that 
$\alpha_k /\rho_k = \mathcal{O}(k^{-2r})$.
Therefore, the sequence $\{\alpha_k /\rho_k\}$ is summable since $r\in (\frac{1}{2}, 1)$.
In view of this and the last display, we conclude that $\{h(x^{k})\}$ is bounded.
Thus, there exists $C_{G}>0$ such that $\|G(x^{k})\| \le C_{G}$ for all $k\ge 0.$
This further yields that $\alpha_k = \Theta(\rho_k^{-1}) = \Theta(k^{-r})$.

By the last display and the facts $\alpha_k \rho_k = \Theta(1)$ and $\alpha_k/\rho_k = \Theta(k^{-2r})$,
there exist positive constants $c_1, c_2$ such that
\begin{align}
\notag h(x^{k+1}) + c_1 h(x^{k+1}) 
& \le  h(x^k) +  c_2 (1+k)^{-2r} 
\quad \forall k\ge 0.
\end{align}
Invoking Lemma~\ref{lm-rate} with $p=1$, $r_1=0$ and $r_2=2r$,  
we obtain that $h(x^{k}) = \mathcal{O}(k^{-2r})$ 
and hence $\|G(x^{k})\| = \mathcal{O}(k^{-{r}})$.
Combining this with the fact $\rho_k = \Theta(k^{r})$, 
we obtain \eqref{eq-wt-mu}.

Next, we prove item (ii). The inequality \eqref{eq-Q-des-1} follows directly from \eqref{eq-Q-des}.
Moreover, by \eqref{eq-wt-mu}, we have $\widetilde{\beta}_{k} \coloneqq \frac{1}{2} (\rho_{k+1} - \rho_k) \| G(x^{k+1})\|^2 
= \mathcal{O}(k^{-(1 + r)}),$ which implies that $M_{\widetilde{\beta}} 
\coloneqq \sum_{k=0}^{\infty} \widetilde{\beta}_k \in [0, \infty)$. 
For every $k\ge 0$, since $\rho_0\le \rho_k$, we have 
\begin{equation*}
    Q_{\rho_0}(x^{k}) \le Q_{\rho_k}(x^{k}) \le Q_{\rho_0}(x^0) + \sum_{l=0}^{k}\widetilde{\beta}_l 
    \le Q_{\rho_0}(x^0) + M_{\widetilde{\beta}},
\end{equation*}
where the second inequality follows from \eqref{eq-Q-des-1}.
Recalling the definition of $\rho_0$ in Assumption~\ref{ass-ec}(vi), we conclude that $\{x^k\}$ is bounded.

Next, we turn to item (iii). For every $k\ge 0$, since $\|G(x^{k})\| \le \widetilde{\mu}/\rho_k$ (see \eqref{eq-wt-mu}), we have
\begin{equation*}
 \frac{\widetilde{\mu}}{2} \|G(x^k)\| \le 
\widetilde\mu \| G(x^{k})\| - \frac{\rho_{k}}{2} \|G(x^k)\|^2 
= E_{\widetilde\mu}(x^{k}) - Q_{\rho_k}(x^{k})
\le \frac{\widetilde\mu^2}{2\rho_{k}},
\end{equation*}
where the second inequality holds because $\sup_{y\ge 0}\{{\widetilde{\mu}}y-(\rho_k/2)y^2\} = \widetilde{\mu}^2/(2\rho_k)$.

Finally, we prove item (iv). 
Let $k\ge 0.$ 
Let $\mu_{k+1} \coloneqq \rho_k \|G(x^{k+1})\|$, 
and let $u^{k+1} \coloneqq G(x^{k+1})/\|G(x^{k+1})\|$ if $G(x^{k+1}) \neq 0$, and $u^{k+1} \coloneqq 0$ otherwise.
Then we have $\mu_{k+1} \le \widetilde{\mu}$ (see \eqref{eq-wt-mu}) and 
$u^{k+1} \in \partial \| \cdot\| (G(x^{k+1}))$. 
Let $\xi^{k+1}$ be defined as in \eqref{eq-xi-ub-ec1}.
Then,
\begin{align}\label{eq-xi-Emu}
     \xi^{k+1} = \nabla f(x^{k+1}) + \mu_{k+1} DG(x^{k+1})^{*} u^{k+1} + \zeta^{k+1}_{\phi} + \zeta^{k+1}_{\mathcal{C}}
\in \partial E_{\mu_{k+1}} (x^{k+1}).
\end{align}
Therefore, for all $x\in \mathcal{C}$,
\begin{align}
\notag & f(x) + \phi(x) + \widetilde\mu\|G(x)\| \ge f(x) + \phi(x) + \mu_{k+1} \|G(x)\|
\\ \notag & \ge f(x^{k+1}) + \phi(x^{k+1}) + \mu_{k+1}\|G(x^{k+1})\|  + \langle \xi^{k+1}, x - x^{k+1} \rangle
- \frac{L_f + \mu_{k+1} L_G}{2} \| x - x^{k+1} \|^2
\\ \notag & \ge f(x^{k+1}) + \phi(x^{k+1})  + \langle \xi^{k+1}, x - x^{k+1} \rangle
- \frac{\widetilde{\varrho}}{2} \| x - x^{k+1} \|^2
\\ \notag & \ge f(x^{k+1}) + \phi(x^{k+1}) + \widetilde{\mu}\|G(x^{k+1})\|  + \langle \xi^{k+1}, x - x^{k+1} \rangle
- \frac{\widetilde{\varrho}}{2} \| x - x^{k+1} \|^2 - \widetilde{\mu}^2 \rho_{k+1}^{-1},
\end{align}
where the second inequality holds because of \eqref{eq-xi-Emu} and the facts that $x^{k+1}\in \mathcal{C}$ and $E_\mu$ is $L_f+\mu L_G$-weakly convex (see the discussion below \eqref{eq-E-rho}), 
the third inequality holds since $\mu_{k+1}\le \widetilde{\mu}$ and $\widetilde{\varrho} = L_f + \widetilde{\mu} L_G$, and the last inequality holds because $\|G(x^{k+1})\| \le \widetilde{\mu}/\rho_{k+1}$  (see \eqref{eq-wt-mu}).
This implies that $\xi^{k+1} \in \partial_{\nu_k}^{\widetilde{\varrho}} E_{\widetilde{\mu}}(x^{k+1})$.

Finally, the inequalities in \eqref{eq-C2-ec} follow directly from \eqref{eq-xi-ub-ec1} and \eqref{eq-wt-mu}.
\end{proof}

We are now ready to establish the convergence of Algorithm~\ref{alg-qpen-pg} with $r\in (2/3,1)$. Specifically, we prove the convergence of $\{x^k\}$ to a KKT point of \eqref{opt-NLP-eq}: we say that $x$ is a KKT point of \eqref{opt-NLP-eq} if there exists $\lambda \in \mathbb{R}^m$ such that
\begin{equation*}
    0 \in \nabla f(x) + \partial \phi(x) + \partial\delta_{\cal C}(x) + D G(x)^*\lambda
    \quad \text{and} \quad
    G(x) = 0;
\end{equation*}
recall that if $x$ is a local minimizer of \eqref{opt-NLP-eq}, then it is a KKT point under a certain standard constraint qualification; see, e.g., \cite[Example~10.8]{RW98}.

We will show that the $\{x^k\}$ generated satisfies (C1), (C2) and (C3) with suitable parameters and hence we can invoke Theorem~\ref{th-main} to establish convergence. To the best of our knowledge, this is the first global sequential convergence and convergence rate analysis for quadratic penalty methods (whose penalty parameters typically go to infinity) based on KL property of the commonly used $\ell_1$-based nonsmooth penalty function given as in \eqref{eq-E-rho}.

\begin{theorem}[Global convergence of Algorithm~\ref{alg-qpen-pg}]
Consider \eqref{opt-NLP-eq} and suppose that Assumption~\ref{ass-ec} holds. 
Let $\{\rho_k\}$ be a nondecreasing positive sequence such that $\rho_k = \Theta(k^{r})$ and $\rho_{k+1}- \rho_k = \mathcal{O}(k^{-(1-r)})$ for some $r\in (\frac{2}{3}, 1)$.
Let $\{x^{k}\}$ be generated by Algorithm~\ref{alg-qpen-pg}. 
Let $E_{\mu}$ and $\widetilde{\mu}$ be given in \eqref{eq-E-rho} and \eqref{eq-wt-mu}, 
respectively. 
Suppose that $E_{\widetilde{\mu}}$ is a KL function. 
Then the following statements hold.
\begin{enumerate}[{\rm (i)}]
    \item $\{x^{k}\}$ converges to a KKT point $x^*$ of \eqref{opt-NLP-eq}, 
and $\sum_{k=0}^{\infty}\|x^{k+1} - x^{k}\| < + \infty.$

\item If $E_{\widetilde{\mu}}$ is a KL function with exponent $\theta \in (0, 1)$,
then
\begin{equation*}
\| x^{k} - x^*\| = \mathcal{O}(k^{-\widetilde{r}}),
\end{equation*}
where
\begin{equation}\label{eq-rate-ec}
\widetilde{r} \coloneqq \widetilde{r}(r) \coloneqq \begin{cases}
\tfrac{3}{2}r - 1 
& \text{if \ } \theta \in (0, \tfrac{1}{2}] , 
\\ \min\left\{\tfrac{3}{2}r - 1, 
\frac{(1- \theta) r}{2 \theta}, 
\frac{(1 - \theta) (1 - r) }{2 \theta - 1} \right\}
& \text{if \ } \theta \in (\tfrac{1}{2}, 1).
\end{cases}
\end{equation}
\end{enumerate}
\end{theorem}

\begin{remark}
Let $\widetilde{r}(r)$ be defined by \eqref{eq-rate-ec}. If $\theta\in(0,\frac12]$, then $\sup\{
\widetilde{r}(r): r\in(\frac{2}{3}, 1)\}
=\frac12;$
if $\theta\in(\frac12,1)$, then $\max
\{ \widetilde{r}(r): r\in(\frac{2}{3},1)\} =\widetilde{r}(r^*)
=\frac{1-\theta}{4\theta-1},$
where $r^*=\frac{2\theta}{4\theta-1} \in (\frac{2}{3},1)$.
\end{remark}

\begin{proof}
  Let $\psi \coloneqq E_{\widetilde{\mu}}$,
$\hat{\psi}_{k} \coloneqq Q_{\rho_k}(x^{k}) 
+ \sum_{l=k}^{\infty} \widetilde{\beta}_l$ with $\widetilde{\beta}_l$ defined as in \eqref{eq-Q-des-1}, 
$\beta_{k} \coloneqq  \frac{1}{2}\widetilde\mu^2\rho_{k}^{-1} + 
\sum_{l=k}^{\infty} \widetilde{\beta}_l$,
$\gamma_k \coloneqq 0$,
$\nu_k \coloneqq \widetilde{\mu}^2/\rho_{k+1}$, and $\xi^{k+1}$ is defined as in \eqref{eq-C2-ec} for all $k\ge 0$.
Then, upon setting $r_{\alpha} = r_{\beta} = r_{\gamma} = r_{\nu} = r$, 
we have $\alpha_k = \Theta(k^{-r_{\alpha}})$ (see Lemma~\ref{lm-E-Q}(i)), 
$\beta_k = \mathcal{O}(k^{-r_{\beta}})$ (see Lemma~\ref{lm-E-Q}(iii)), 
$\gamma_k = \mathcal{O}(k^{-r_{\gamma}})$, 
and $\nu_k = \mathcal{O}(k^{-r_{\nu}})$.
Moreover, the sequence $\{x^{k}\}$ satisfies {\rm (C\ref{C1})} by \eqref{eq-Q-des-1} and \eqref{eq-E-Q}, 
satisfies {\rm (C\ref{C2})} by \eqref{eq-C2-ec}, and satisfies {\rm (C\ref{C3})} because $\{x^{k}\} \subseteq \mathcal{C}$ is bounded (see Lemma~\ref{lm-E-Q}(ii)) and $f,\phi,G$ are continuous.

Note that $r_{\tau} \coloneqq \min\{r_{\gamma}, \frac{\tau - 1}{\tau} r_{\beta}, \tfrac{1}{2} r_{\nu}\} = \frac{1}{2}r$
and $r_{\alpha} + r_{\tau} - 1 = \frac{3}{2}r-1 > 0$ for all $\tau \ge 2$. 
Thus, parameters $\{\alpha_k\}$, $\{\beta_{k}\}$, $\{\nu_k\}$ and $\{\gamma_{k}\}$ satisfy condition {\rm (A\ref{(B)})} of Theorem~\ref{th-main} with $\widetilde{\tau} =2$.
By Theorem~\ref{th-main}, $\{x^{k}\}$ converges to a critical point $x^*$ of $\psi = E_{\widetilde{\mu}}$, and $\sum_{k=0}^{\infty}\|x^{k+1} - x^{k}\| < + \infty.$
Since $\lim_{k\to\infty} \|G(x^{k})\| = 0$ (see \eqref{eq-wt-mu}), we have $G(x^*) = 0$, 
which implies that $x^*$ is a KKT point of \eqref{opt-NLP-eq}.
Finally, item (ii) follows directly from Theorem~\ref{th-main}(iii) and (iv) with $\tau = \widetilde{\tau} = 2.$
\end{proof}

\subsection{Application III: Inexact level constrained proximal gradient method}\label{sec3.3}

In this subsection, we consider the following constrained composite optimization problem:
\begin{align}\label{opt-compo-constr}
\min_{x \in \mathbb{R}^n} &\ \ \psi_0 (x) \coloneqq f_0(x) + \chi_0(x) \notag
\\ \text{s.t.} &\ \ \psi_i(x) \coloneqq f_i(x) + \chi_i(x) \leq \eta_i \quad \forall i = 1,\ldots,m,
\end{align}
where $\eta=(\eta_1,\ldots,\eta_m) \in \mathbb{R}^{m}$, and for each $i = 0,1,\ldots,m$, the function $\psi_i$ is given by the sum of a smooth function $f_i$ and a convex function $\chi_i$, and we impose the following assumption, following \cite[Assumptions 1, 2 and 3]{25BDL}. Below, for convenience, we use $[m]$ to denote the index set $\{1,\ldots,m\}$.
\begin{assumption}
\label{ass-lcpg}
In \eqref{opt-compo-constr}, assume that its optimal value is finite, and that 
\begin{enumerate}[{\rm (i)}]
\item for each $i=0,1,\ldots,m$, $f_i:\mathbb{R}^{n} \to \mathbb{R}$ is continuously differentiable and 
$L_i$-smooth on $\dom (\chi_0)$ with constant $L_i>0$, i.e., 
\begin{equation*}
\| \nabla f_i (x) - \nabla f_i(y) \| \le L_i \| x - y \| \quad \forall x, y \in \dom \chi_0;
\end{equation*}
moreover, $\chi_0:\mathbb{R}^n\to (-\infty,\infty]$ is proper, closed and convex, and each $\chi_i$, $i=1,\ldots,m$, is convex and real-valued;
furthermore, the feasible set of \eqref{opt-compo-constr}, i.e.,
$\mathcal{F} \coloneqq \left\{ x \in \mathbb{R}^n : \psi_i(x) \leq \eta_i, \; i=1,\ldots,m \right\} \cap \dom \chi_0$, is nonempty and compact;
\item (strict feasibility) there exist $x^{0} \in \dom  \chi_{0}$ and $\eta^{0}\in \mathbb{R}^{m}$ such that
\begin{equation*}
f_i(x^0) + \chi_i(x^0) < \eta^{0}_i < \eta_i \quad \forall i \in [m];
\end{equation*}
\item  the following Mangasarian-Fromovitz constraint qualification (MFCQ) holds at every $x\in \mathcal{F}$: 
there exists $z \in\mathcal{T}_{{\mathrm dom} {\chi_0}}(x)$ such that
\begin{equation*}
\max_{v \in \partial \psi_i(x)}\langle v, z\rangle < 0 \quad \forall i \in \mathcal{A}(x), 
\end{equation*}
where $\mathcal{A}(x) \coloneqq \{ i \in [m] : \psi_i(x) = \eta_i\}$, $\mathcal{T}_{{\mathrm dom} {\chi_0}}(x) := {\rm cl}\left(\bigcup_{t\ge 0}t({\rm dom}\,\chi_0 - x)\right)$ and  ${\rm cl}(C)$ denotes the closure of the set $C$.
\end{enumerate}
\end{assumption}

A classical method for solving \eqref{opt-compo-constr} under the above assumption (and assuming in addition that $\chi_i\equiv 0$ for $i = 0,\ldots,m$) is the moving balls approximation (MBA) method first introduced in \cite{10AST}. Global sequential convergence and the convergence rate of many variants of MBA methods have been established, assuming that the subproblems are {\em exactly} solved and a special potential function satisfies the KL property; see \cite{16BP,21YPL,20BCP}. 
Here, we focus on a recent variant of MBA methods known as the inexact level constrained proximal gradient (LCPG) method, proposed in \cite{25BDL}. This method modifies the subproblems of MBA method by slightly perturbing their feasible sets, and allows these modified subproblems to be {\em inexactly} solved. Its global sequential convergence property and local convergence rate are currently unknown. Here, we show that this method fits in our framework for some weakly convex function $\psi$ and hence Theorem~\ref{th-main} can be invoked to derive its convergence properties.

The main idea of inexact LCPG method is to construct a sequence of relatively easier subproblems 
based on some convex surrogate functions $\psi^{k}_i$ ($i=0,1,\ldots, m$) and variable constraint levels $\eta^{k}$:
\begin{align}
\label{opt-sub-lcpg}
& \min_{x\in \mathbb{R}^{n}}\!
\left\{ \psi_{0}^{k}(x) : \psi_{i}^{k}(x) \le \eta^{k}_i,\ i=1,\ldots, m\right\},
\end{align}
where, for each $i = 0, 1,\ldots, m$, the function $\psi_i^k$ is defined by
\begin{equation}\label{eq-lcpg-psi-ik}
\psi_i^k(x) \coloneqq f_i(x^k)+\langle\nabla f_i(x^k), x-x^k\rangle+\tfrac{L_i+\ell_i}{2}\|x-x^k\|^2+\chi_i(x)\ \ \forall x \in \mathbb{R}^n;
\end{equation}
here, $\ell_{0} > 0$ is a parameter and $\ell_{i} = 0$ for all $i\in [m]$.\footnote{Here, we introduce an additional parameter $\ell_{0}>0$ to simplify our analysis. Although the original LCPG method takes $\ell_0 =0$, our formulation is equivalent to applying the LCPG method with $L_0+\ell_0$ in place of $L_0$ as the Lipschitz constant of $\nabla f_0$.} The inexact LCPG method (see \cite[Section~6]{25BDL}) is presented in Algorithm~\ref{alg-lcpg}.

\begin{algorithm}[h!]
\caption{Inexact level constrained proximal gradient (LCPG) method for \eqref{opt-compo-constr}} 
\label{alg-lcpg}
\begin{algorithmic}[0]
\setlength{\columnwidth}{\linewidth}
\State {\bf Require.} $x^0$, $\eta^{0}$ and $L_i$, $i=0,\ldots,m$, as in Assumption~\ref{ass-lcpg}, $\ell_0 > 0$ and $\ell_i = 0$ for all $i\in [m]$.
\For{$k=0,1,\ldots$} 
\State Choose $\delta^k$ such that $\delta^k_i\in (0,\eta_i-\eta_i^k)$ for all $i\in [m]$, pick $\varepsilon_k \in [0,\min_{i\in [m]}\delta^k_i)$.
\State Compute $x^{k+1}$, an inexact solution of \eqref{opt-sub-lcpg}, such that
\begin{align}\label{eq-inex-1}
& \psi_0^k(x^{k+1})-\psi_0^k(\widetilde{x}^{k+1}) \leq \varepsilon_k, \quad \|[\psi^k(x^{k+1}) - \eta^k]_{+}\| \leq \varepsilon_k,
\\ & \mathcal{L}_k(x^{k+1}, \widetilde{\lambda}^{k+1}) \leq \mathcal{L}_k(\widetilde{x}^{k+1}, \widetilde{\lambda}^{k+1})+\varepsilon_k,
\label{eq-inex-2}
\end{align}
\State where $\widetilde{x}^{k+1}$ is an exact solution of \eqref{opt-sub-lcpg}, $\widetilde\lambda^{k+1}\in \R^m_+$ is a Lagrange
 multiplier, 
 \State and 
\begin{equation}\label{eq-Lag-k}
    \mathcal{L}_{k}(x, \lambda) \coloneqq 
    \psi_{0}^{k}(x) +  \sum_{i=1}^{m} \lambda_{i}(\psi_{i}^{k}(x) - \eta^{k}_i)
    \quad \forall (x, \lambda) \in \mathbb{R}^n\times\mathbb{R}^m.
\end{equation}
\State Set $\eta^{k+1}=\eta^k+\delta^k$.
\EndFor
\end{algorithmic}
\end{algorithm}

We summarize some fundamental properties of the inexact LCPG method in Lemma~\ref{lm-lcpg-wd} below.
The claim on strict feasibility of $x^k$ for \eqref{opt-sub-lcpg} in item (ii) can be seen from the first display of the proof of \cite[Theorem~7]{25BDL}.
From this, together with an argument analogous to the proof of \cite[Proposition 2]{25BDL}, items (i) and (ii) follow. 
Items (iii) and (iv) are derived from \cite[Theorem~7]{25BDL}.
\begin{lemma}
\label{lm-lcpg-wd}
Consider \eqref{opt-compo-constr} and suppose that Assumption~\ref{ass-lcpg} holds, 
and that $\lim_{k\to\infty}\eta^k = \eta$. 
Then the following statements hold for Algorithm~\ref{alg-lcpg}.
\begin{enumerate}
[{\rm (i)}] 
\item The sequence $\{x^k\}$ is well-defined and feasible for problem \eqref{opt-compo-constr}; in particular, this means, for each $k\ge 0$, a solution $\widetilde x^{k+1}$ of \eqref{opt-sub-lcpg} and a Lagrange multiplier $\widetilde \lambda^{k+1}$ exist. 
\item For each $k\ge 0$, $x^k$ is strictly feasible for \eqref{opt-sub-lcpg}; moreover, the exact solution $\widetilde{x}^{k+1}$ of \eqref{opt-sub-lcpg} is its unique solution, and together with the $\widetilde{\lambda}^{k+1}\in \mathbb{R}^{m}_{+}$, it holds that
\begin{align}\label{opt-lcpg-lag}
& \widetilde{x}^{k+1} = \argmin_{x\in \R^n} \mathcal{L}_k(x, \widetilde{\lambda}^{k+1}),
\end{align}
and 
\begin{equation}\label{eq-lcpg-sub-kkt}
    0 \in \partial_x \mathcal{L}_k(\widetilde{x}^{k+1}, \widetilde{\lambda}^{k+1}), \ \ 
\widetilde{\lambda}_i^{k+1}\left(\psi_i^k(\widetilde{x}^{k+1})-\eta_i^k\right)  =0 \quad \forall i \in[m],
\end{equation}
where $\mathcal{L}_k$ is the Lagrangian function associated with \eqref{opt-sub-lcpg}, defined as in \eqref{eq-Lag-k}.
\item There exists a constant $B>0$ such that $\sup\{\|\widetilde{\lambda}^{k+1}\|: k\in \mathbb{N}_0\} \le B$.

\item  Every accumulation point $\bar{x}$ of $\{x^k\}$ is a KKT point of \eqref{opt-compo-constr}, i.e., 
$\bar x$ is feasible for \eqref{opt-compo-constr} and there exists $\bar \lambda\in \mathbb{R}^m_+$ such that
\begin{align*}
    & 0 \in  \nabla f_{0}(\bar x) + \partial \chi_0(\bar x) + \sum_{i=1}^{m} \bar \lambda_i ( \nabla f_{i}(\bar x) + \partial \chi_i(\bar x)),
    \\ & \bar \lambda_i (  f_{i}(\bar x) + \chi_i(\bar x) - \eta_i) = 0 \quad \forall i\in [m].
\end{align*}
\end{enumerate}
\end{lemma}
We establish in Lemma~\ref{lm-lcpg-Cs} below that the $\{x^k\}$ generated by the inexact LCPG method 
can be interpreted as a minimizing sequence of the following penalty function:
\begin{equation}\label{eq-lcpg-Phi-beta}
F_{\beta} (x) \coloneqq \psi_{0}(x) + \beta \sum_{i=1}^{m}[\psi_i(x) - \eta_i]_{+} \quad \forall x \in \mathbb{R}^{n},
\end{equation}
where $\beta>0$ is a sufficiently large penalty parameter, and $\psi_i$'s are given in \eqref{opt-compo-constr}.
\begin{lemma}\label{lm-lcpg-Cs}
Consider \eqref{opt-compo-constr} and suppose that Assumption~\ref{ass-lcpg} holds, and that $\lim_{k\to\infty}\eta^k = \eta$.	
Let $\{x^{k}\}$ be generated by Algorithm~\ref{alg-lcpg}, 
with each iterate satisfying \eqref{eq-inex-1} and \eqref{eq-inex-2}.
Let $B$ be given in Lemma~\ref{lm-lcpg-wd}(iii), and let $F_{\beta}$ be defined by \eqref{eq-lcpg-Phi-beta} with a parameter $\beta\ge B$.
Let $\varrho \coloneqq 4(L_0 + \ell_0+ \beta \sum_{i=1}^m L_i)$.
Then the following statements hold.
\begin{enumerate}[{\rm (i)}]
\item $F_{\beta}$ is $\varrho$-weakly convex. 

\item For each $k\ge 0$, $F_{\beta}(x^{k+1}) + \| \eta - \eta^{k+1} \|_1 \le F_{\beta}(x^{k}) + \| \eta - \eta^{k} \|_1 - \frac{\ell_0}{2}\| x^{k+1} - x^{k}\|^2$.
\item For each $k\ge 0$, it holds that $0 \in \partial_{\widetilde{\nu}_{k}}^{ \varrho} F_{\beta}(x^{k+1})$, 
where
\begin{equation}\label{def-nuk-lcpg}
    \widetilde{\nu}_{k} \coloneqq 
\frac{\varrho}{\ell_0}  (F_{\beta}(x^{k}) + \| \eta - \eta^{k} \|_1 - F_{\beta}(x^{k+1}) - \| \eta - \eta^{k+1} \|_1 ) + (1 + \beta )\|\eta - \eta^k\|_{1}\ge 0.
\end{equation}
\item The sequence $\{x^{k}\}$ is bounded. 
Moreover, for any convergent subsequence $\{x^{k_j}\}$ with $\lim_{j\to\infty}x^{k_j} = \bar{x}$ for some $\bar x\in \R^n$, it holds that $\lim_{j\to\infty}F_{\beta}(x^{k_j}) = F_{\beta}(\bar{x}) \in \mathbb{R}$.
\end{enumerate}
\end{lemma}
\begin{proof}
Let $\widetilde \varrho \coloneqq L_0 + \beta \sum_{i=1}^m L_i$. Then, we have
\begin{equation*}
F_{\beta}(\cdot)  + \tfrac{\widetilde\varrho}{2}\|\cdot\|^2 
= \psi_0(\cdot) + \tfrac{L_0}{2}\|\cdot\|^2 + \beta \sum_{i=1}^{m} 
\max\{ \tfrac{L_i}{2}\|\cdot\|^2 , \psi_i(\cdot) - \eta_i + \tfrac{L_i}{2}\|\cdot\|^2\},
\end{equation*}
which is convex because each function $\psi_i(\cdot) + \tfrac{L_i}{2}\|\cdot\|^2$ is convex on $\dom (\chi_0)$ according to Assumption~\ref{ass-lcpg}(i).
This shows that $F_{\beta}$ is $\widetilde \varrho$-weakly convex.
Since $\varrho \ge \widetilde{\varrho}$, we obtain item (i).

Since the sequence $\{x^{k}\}$ is feasible for \eqref{opt-compo-constr} (see Lemma~\ref{lm-lcpg-wd}(i)), we have
\begin{equation}\label{eq-F=psi}
    F_{\beta}(x^k) = \psi_0(x^k) \quad \forall k\ge 0.
\end{equation}
Moreover, for all $k\ge 0$ and $z \in \dom \chi_0$,
\begin{align}
    \notag \psi_0(x^{k+1}) + \frac{\ell_0}{2}\| x^{k+1} - x^{k}\|^2 
     & \overset{\mathrm{(a)}}{\le}  \psi_0^k(x^{k+1}) \overset{\mathrm{(b)}}{\le} \psi_0^k(\widetilde x^{k+1}) + \varepsilon_k  \notag
\\ & \overset{\mathrm{(c)}}{=} \mathcal{L}_{k}(\widetilde x^{k+1}, \widetilde{\lambda}^{k+1})  + \varepsilon_k
\overset{\mathrm{(d)}}{\le}\mathcal{L}_{k}(z, \widetilde{\lambda}^{k+1}) + \varepsilon_k, 
\label{eq-inex-psi}
\end{align}
where (a) follows from the $L_0$-smoothness of $f_0$ on $\dom (\chi_0)$ and the definition of $\psi^k_0$ in \eqref{eq-lcpg-psi-ik}, 
(b) follows from  \eqref{eq-inex-1}, 
(c) follows from the complementary slackness condition in \eqref{eq-lcpg-sub-kkt},
and (d) follows from \eqref{opt-lcpg-lag}.

Now, we consider item (ii). Setting $z = x^{k}$ in the last display yields
\begin{align*}
    \psi_0(x^{k+1}) + \frac{\ell_0}{2}\| x^{k+1} - x^{k}\|^2
    \le \psi_0(x^{k}) + \sum_{i=1}^{m} \widetilde{\lambda}^{k+1}_{i} ( \psi_i^k(x^{k}) - \eta^{k}_i ) + \varepsilon_k
    \le \psi_0(x^{k}) + \varepsilon_k,
\end{align*}
where the last inequality holds because $x^{k}$ is feasible for \eqref{opt-sub-lcpg} (see Lemma~\ref{lm-lcpg-wd}(ii)).
Moreover, for any $k\ge 0$, since $\eta^{k} < \eta^k + \delta^k = \eta^{k+1}<\eta$, we have
\begin{equation}\label{eq-eps-eta}
    \varepsilon_k < \min_{i\in [m]}\delta^k_i \le 
     \|\delta^{k}\|_1 = \| \eta - \eta^{k}\|_1 - \| \eta - \eta^{k+1} \|_1.
\end{equation}
Combining the last two displays with \eqref{eq-F=psi}, we obtain item (ii).

Next, we turn to item (iii).  
Let $k\ge 0$. Clearly, $\widetilde{\nu}_{k}\ge 0$ thanks to item (ii). Next, from \eqref{eq-F=psi} and \eqref{eq-inex-psi}, it follows that for all $z\in \dom \chi_0$,
\begin{align}
 & F_{\beta}(x^{k+1}) \le \mathcal{L}_{k}(z, \widetilde{\lambda}^{k+1}) + \varepsilon_k \notag
\\ & = f_{0}(x^{k}) + \langle \nabla f_0(x^{k}), z - x^{k}\rangle + \frac{L_0+\ell_0}{2} \| z - x^{k}\|^2  + \chi_0(z)  \notag
\\ & \ \ + \sum_{i=1}^{m} \widetilde \lambda^{k+1}_i 
\left(f_{i}(x^{k}) + \langle \nabla f_i(x^{k}), z - x^{k}\rangle + \frac{L_i}{2} \| z - x^{k}\|^2  + \chi_i(z) - \eta_i^k\right)
+ \varepsilon_k \notag
\\ & \overset{\mathrm{(a)}}{\le} f_0(z) + (L_0+\ell_0)\| z - x^{k}\|^2  + \chi_0(z) +  \sum_{i=1}^{m} \widetilde \lambda^{k+1}_i 
\left(f_i(z) + L_i \| z - x^{k}\|^2  + \chi_i(z) - \eta_i^k\right)
+ \varepsilon_k \notag
\\ & = \psi_0(z) +  \sum_{i=1}^{m} \widetilde \lambda^{k+1}_i 
\left(\psi_i(z) - \eta_i + \eta_i - \eta_i^k\right) 
 + \left(L_0 + \ell_0 +  \sum_{i=1}^{m} \widetilde \lambda^{k+1}_i L_i\right) \| z - x^{k}\|^2 
+ \varepsilon_k \notag
\\ & \overset{\mathrm{(b)}}{\le} \psi_0(z) +  \beta \sum_{i=1}^{m}
[\psi_i(z) - \eta_i]_+ + \beta \|\eta - \eta^k\|_1 
 + \left(L_0 + \ell_0+ \beta \sum_{i=1}^m L_i \right) \| z - x^{k}\|^2 
+ \varepsilon_k \notag
\\ & \overset{\mathrm{(c)}}{\le}  F_{\beta}(z)  + \frac{{{\varrho}}}{2} \| z - x^{k+1}\|^2 + \frac{{{\varrho}}}{2} \| x^{k+1} - x^{k}\|^2  
+ (\beta+1) \|\eta - \eta^k\|_1 \notag
\\ & \overset{\mathrm{(d)}}{\le} 
F_{\beta}(z)  + \frac{{{\varrho}}}{2} \| z - x^{k+1}\|^2 + \widetilde{\nu}_{k},\notag
\end{align}
where (a) follows from the $L_i$-smoothness of $f_i$ on $\dom (\chi_0)$,\footnote{Specifically, we applied the inequality $-f_i(z)\le -f_i(x^k) - \langle \nabla f_i(x^k),z-x^k\rangle + \frac{L_i + \ell_i}{2}\|z-x^k\|^2$.} 
(b) follows from Lemma~\ref{lm-lcpg-wd}(iii) and the condition that $\beta\ge B$, 
(c) follows from ${\varrho} = 4(L_0 + \ell_0+ \beta \sum_{i=1}^m L_i)$, \eqref{eq-eps-eta}, 
and the fact that $\| u + v\|^2 \le 2\|u\|^2 + 2\|v\|^2$ for all $u,v\in \mathbb{R}^n$, 
and (d) holds by item (ii) and the definition of $\widetilde{\nu}_k$ in \eqref{def-nuk-lcpg}. 
This implies item (iii).
 
(iv) By Lemma~\ref{lm-lcpg-wd}(i) and Assumption~\ref{ass-lcpg}(i), the sequence $\{x^{k}\}\subseteq \mathcal{F}$ is bounded. 
Let $\{x^{k_j}\}$ be a convergent subsequence and let $\bar x$ denote its limit. 
Then we see from item (ii), the closedness of $\chi_0$ and the continuity of $\chi_i$, $i\in [m]$, that
\[
F_\beta(\bar x)\le \liminf_{j\to \infty} F_\beta(x^{k_j})\le F_\beta(x^0) + \|\eta - \eta^0\|_1 < \infty,
\]
showing that $\bar x\in {\rm dom}\,\chi_0$, and hence $F_\beta(\bar x)\in \R$. Next, using the fact that $0 \in \partial_{\widetilde{\nu}_{k}}^{ \varrho} F_{\beta}(x^{k+1})$ (see item (iii)), we have for all $j\ge 0$,
\[
f_0(x^{k_j})+\chi_0(x^{k_j}) = F_\beta(x^{k_j})\le F_\beta(\bar x) + \frac{{{\varrho}}}{2} \| \bar x - x^{k_j}\|^2 + \widetilde{\nu}_{k_j-1}.
\]
Passing to the limit as $j\to\infty$ on both sides of the above inequality, and using the assumption $\lim_{k\to\infty}\eta^k = \eta$ (which, together with item (ii) and the fact that $\{x^k\}\subseteq \dom \chi_0$ is bounded, implies $\widetilde{\nu}_{k}\to 0$), we obtain that
$\limsup_{j\to\infty} \chi_0(x^{k_j}) \le \chi_0(\bar{x})$.  
Hence $\lim_{j\to\infty} \chi_0(x^{k_j}) = \chi_0(\bar{x})$ in view of the closedness of $\chi_0$.
This, together with the definition of $F_{\beta}$,
implies item (iv).
\end{proof}

We are now ready to establish the convergence of the sequence generated by the inexact LCPG method under suitable choices of $\{\delta^k$\} based on the KL property and KL exponent of the penalty function $F_\beta$ in \eqref{eq-lcpg-Phi-beta}. We would like to mention that previous convergence analysis for these kinds of methods relied on the KL property of a special potential function which is derived by replacing each $\psi_i$, $i=0,\ldots,m$, with a quadratic majorant, whose KL property and exponent can be relatively harder to analyze; see, e.g., \cite{16BP,21YPL,25LPB}.

\begin{theorem}[Global convergence of inexact LCPG]
Consider \eqref{opt-compo-constr} and suppose that Assumption~\ref{ass-lcpg} holds, and that the sequence $\{\delta^k\}$ in Algorithm~\ref{alg-lcpg} is chosen according to either {\rm (S1)} or {\rm (S2)} below.
\begin{enumerate}[{\bf \rm (S1)}]
\item Pick $q_{\delta}\in (0, 1)$ and set $\delta^k = (\eta - \eta^0)(1 - q_{\delta}) q^k_{\delta}$ for all $k\ge 0$.
\item Pick $r_{\eta}>2$ and set $\delta^k = (\eta - \eta^0) \left( (1+k)^{-r_{\eta}} - (2+k)^{-r_{\eta}} \right)$ for all $k\ge 0$.
\end{enumerate}
Let $\varepsilon_k = \frac{1}{2} \min_{i\in [m]}\delta^k_i$ for all $k\ge 0$.
Let $\{x^{k}\}$ be generated by the inexact LCPG method, 
with each iterate satisfying \eqref{eq-inex-1} and \eqref{eq-inex-2}.
Let $B$ be given in Lemma~\ref{lm-lcpg-wd}(iii) and $\beta \ge B$.
Suppose that $F_{\beta}$ defined in \eqref{eq-lcpg-Phi-beta} is a KL function. Then the following statements hold. 
\begin{enumerate}[{\rm (i)}]
    \item $\{x^{k}\}$ converges to a KKT point $x^*$ of \eqref{opt-compo-constr}, 
and $\sum_{k=0}^{\infty}\|x^{k+1} - x^{k}\| < + \infty$.

\item If $F_{\beta}$ is a KL function with exponent $\theta \in (0, 1)$, then
\begin{enumerate}[{\rm (a)}]
        \item {\rm (Linear convergence)} if $\theta \in (0, \frac{1}{2}]$ and $\{\delta^k\}$ is chosen according to {\rm (S1)}, 
    then there exists $q \in (0, 1)$ such that
    \begin{equation*}
    	\| x^{k}  - x^*\| = \mathcal{O}(q^{k});
    \end{equation*}
    \item {\rm (Sublinear convergence)} if $\{\delta^k\}$ is chosen according to {\rm (S2)}, then
    \begin{equation*}
    	\| x^{k}  - x^*\| = \mathcal{O}(k^{-\widetilde{r}})
        \quad \text{with} \quad 	\widetilde{r} \coloneqq \begin{cases}
     \frac{r_{\eta}}{2} - 1 & \text{if } \theta \in (0, \tfrac{1}{2}], 
	 \\ \min\left\{\frac{r_{\eta}}{2} - 1, 
    \frac{(1-\theta)r_{\eta}}{2\theta}, \frac{1-\theta}{2\theta - 1} \right\}
    & \text{if } \theta \in (\tfrac{1}{2}, 1).
	\end{cases}
    \end{equation*}
\end{enumerate}    
\end{enumerate}

\end{theorem}
\begin{proof}
For {\rm (S1)}, we have $\|\eta - \eta^k\|_1 = \|\eta - \eta^0\|_1 q^k_{\delta}$ for all $k\ge 0$.
For {\rm (S2)}, we have $\|\eta - \eta^k\|_1 = \|\eta - \eta^0\|_1 (1+k)^{-r_{\eta}}$ for all $k\ge 0$.
For both choices, $\lim_{k\to\infty}\eta^{k} = \eta$.

Setting $\psi = F_{\beta}$, ${\varrho} = 4(L_0 + \ell_0 + \beta \sum_{i=1}^m L_i)$, 
and let $\alpha_{k} = 1/\ell_0$, $\beta_{k} = \| \eta - \eta^{k}\|_1$, 
$\gamma_{k} = 0$, $\nu_k = (1 + \beta) \|\eta - \eta^k\|_{1}$, 
and $\hat{\psi}_k = F_{\beta}(x^{k}) + \| \eta - \eta^{k}\|_1$, for all $k\ge 0$.
It follows from Lemma~\ref{lm-lcpg-Cs} that $\{x^{k}\}$ satisfies conditions {\rm (C\ref{C1})}, {\rm (C\ref{C2})}, and {\rm (C\ref{C3})} with $\xi^{k+1} = 0$.

For {\rm (S1)}, condition {\rm (A\ref{(A)})} of Theorem~\ref{th-main} holds.
On the other hand, suppose that $\{\delta^{k}\}$ is chosen according to {\rm (S2)},
then $\beta_k = \|\eta - \eta^k\|_1  = \mathcal{O}( k^{-r_{\eta}})$ and $\nu_k = (1 + \beta) \|\eta - \eta^k\|_{1} = \mathcal{O}( k^{-r_{\eta}} )$.
Moreover, $\alpha_k = \Theta(k^{-r_{\alpha}})$ and $\gamma_{k} = 0 = \mathcal{O}(k^{-r_{\gamma}})$, with $r_{\alpha} = 0$ and $r_{\gamma} = r_{\eta}$.
Then, we have, upon noting $r_\eta > 2$, that
\[
r_{\tau} = \min\left\{r_{\gamma}, \frac{\tau - 1}{\tau}r_{\beta}, \frac{1}{2} r_{\nu}\right\} = \min\left\{r_{\eta}, \frac{\tau - 1}{\tau}r_{\eta}, \frac{1}{2} r_{\eta}\right\} = \frac{1}{2} r_{\eta}\ \ {\rm and}\ \ 
r_{\alpha} + r_{\tau} - 1>0\ \ \forall \tau \ge 2.
\]
In this setting, condition {\rm (A\ref{(B)})} of Theorem~\ref{th-main} holds with $\widetilde{\tau} =2.$

Then, item (i) follows directly from Theorem~\ref{th-main}(i), 
together with the fact that every accumulation point $\bar{x}$ of $\{x^k\}$ is a KKT point of \eqref{opt-compo-constr} (see Lemma~\ref{lm-lcpg-wd}(iv)).
Item (ii)(a) follows from Theorem~\ref{th-main}(ii).
For item (ii)(b), the case $\theta\in (0, \frac{1}{2}]$ follows from \eqref{rate-main1}, 
and the case $\theta \in (\frac{1}{2},1)$ follows from \eqref{rate-main2} with $\tau = 2$.
\end{proof}

\subsection{Application IV: Variable smoothing method}\label{sec3.4}
Variable smoothing methods have been proposed for handling objectives that involve proximable functions composed with linear maps; see, e.g., \cite{15BH,20BB,BW21} and references therein.
Here, we follow \cite{BW21} and consider the following optimization problem:
\begin{equation}\label{opt-fgA}
\min_{x\in \mathbb{R}^n}\, \psi(x) \coloneqq f(x) + g(Ax),
\end{equation} 
where $f$, $g$ and $A$ satisfy the following assumption.
\begin{assumption}\label{ass-fgA}
In problem \eqref{opt-fgA}, 
\begin{enumerate}[{\rm (i)}]
\item the function $f:{\mathbb R}^n \rightarrow \mathbb{R}$
is $L_{f}$-smooth with constant $L_f>0$;

\item the function $g :{\mathbb R}^m \rightarrow \mathbb{R}$ 
is $\varrho$-weakly convex and $M_{g}$-Lipschitz continuous with constants $\varrho> 0$ and $M_{g}>0$; 

\item the matrix $A\in \mathbb{R}^{m\times n}$ is nonzero;

\item the level set $\Omega_0 \coloneqq \{ x : \psi(x) \le \psi(x^0) + M_g^2 (2\varrho)^{-1} \}$ is bounded for some $x^0 \in \mathbb{R}^n$.
\end{enumerate}
\end{assumption}
\begin{remark}
Assumptions~\ref{ass-fgA}(i)--(ii) were imposed in \cite{BW21} to establish the iteration complexity of the variable smoothing method for finding an approximate stationary point; see \cite[Theorem~4.1]{BW21}. To further ensure boundedness of the iterates, we additionally impose the level-boundedness condition in Assumption~\ref{ass-fgA}(iv).
\end{remark}

\begin{remark}[Weak convexity of $\psi$ in \eqref{opt-fgA}]\label{rm-fgA-wc}
Consider \eqref{opt-fgA} and suppose that Assumption~\ref{ass-fgA} holds.
We recall below that the $\psi$ in \eqref{opt-fgA} is $\widetilde{\varrho}$-weakly convex with $\widetilde{\varrho} \coloneqq L_f + \varrho\|A\|^2$.

Indeed, let $x \in \mathbb{R}^n$ and $\xi \in \partial \psi(x)$.
Since $\partial \psi(x) \subseteq \nabla f(x) + A^* \partial g (Ax)$ by \cite[Theorem~10.6]{RW98}, 
there exists $\zeta \in \partial g(Ax)$ such that $\xi = \nabla f(x) + A^*\zeta$.
Moreover, by the $\varrho$-weak convexity of $g$ and the $L_f$-smoothness of $f$, for any $z \in \mathbb{R}^n$,
\begin{align*}
    f(z) + g(Az) & \ge f(x) + \langle \nabla f(x), z - x\rangle - \tfrac{L_f}{2}\| z - x\|^2 + g(Ax)  + \langle \zeta, Az - Ax\rangle - \tfrac{\varrho}{2} 
    \| A(z - x) \|^2
    \\& \ge f(x) + g(Ax) + \langle \xi, z - x\rangle - \frac{L_f + \varrho\|A\|^2}{2} \| z - x \|^2,
\end{align*}
which, together with \cite[Lemma 2.1]{19DavisDrusvyatskiy},
implies that $\psi$ is $\widetilde{\varrho}$-weakly convex.
\end{remark}

The variable smoothing method in \cite{BW21} for \eqref{opt-fgA} approximates 
$g$ by its Moreau envelope $g_{\mu}$ with parameter $\mu\in (0, \varrho^{-1})$, which is defined as
\begin{equation}\label{eq-def-gmu-ME}
g_{\mu}(y) \coloneqq \inf_{w \in \mathbb{R}^m}\bigg\{ g(w) + \frac{1}{2\mu} 
\| w - y \|^2\bigg\} \quad \forall y\in \mathbb{R}^m.
\end{equation}
It also makes use of the proximal operator of the function $\mu g$ with parameter $\mu\in (0, \varrho^{-1})$, which is defined as 
\begin{equation*}
    \mathrm{prox}_{\mu g}(y) \coloneqq \argmin_{w \in \mathbb{R}^m}\bigg\{ g(w) + \frac{1}{2\mu} 
\| w - y \|^2\bigg\} \quad \forall y\in \mathbb{R}^m;
\end{equation*} 
here, the set of minimizers is a singleton because the function $g(\cdot) + \|\cdot - y\|^2/(2\mu)$ is strongly convex when $\mu\in (0, \varrho^{-1})$, thanks to the $\varrho$-weak convexity of $g$.
In the next lemma, we summarize some properties of the Moreau envelopes of weakly convex functions. 
\begin{lemma}\label{lm-wc-Moreau}
Let $g :{\mathbb R}^m \rightarrow \mathbb{R}$ 
be $\varrho$-weakly convex and $M_{g}$-Lipschitz continuous for some $\varrho> 0$ and $M_{g}>0$. Then the following statements hold.
\begin{enumerate}[{\rm (i)}]
    \item for every $\mu\in (0, \varrho^{-1})$, 
        the function $g_{\mu}$ defined in \eqref{eq-def-gmu-ME} is continuously differentiable on $\mathbb{R}^m$, and satisfies, for all $y\in \mathbb{R}^m$,
        \begin{align}\label{nabl gmu}
        & \nabla g_{\mu} (y) = \tfrac{1}{\mu} ( y - \mathrm{prox}_{\mu g} (y) )
        \in \partial g( \mathrm{prox}_{\mu g} (y) ) 
        \quad \text{and} \quad \|\nabla g_{\mu}(y)\| \le M_g;
        \end{align}

     \item for every $\mu\in (0, \varrho^{-1})$, the gradient $\nabla g_{\mu}$ is $\max\{\mu^{-1}, \frac{\varrho}{1-\varrho \mu} \}$-Lipschitz continuous;

    \item for any $0< \mu_2 < \mu_1 < \varrho^{-1}$, 
    \begin{equation}\label{gmu12}
        g_{\mu_2}(y) \le g_{\mu_1}(y) + \frac{\mu_1}{2\mu_2} M_g^2 (\mu_1 - \mu_2) 
        \quad \forall y \in \mathbb{R}^m;
    \end{equation}
    
    \item for every $\mu\in (0, \varrho^{-1})$, $g_{\mu}$ is $\frac{\varrho}{1-\varrho \mu}$-weakly convex;

    \item for any $\mu \in (0, \varrho^{-1})$ and $y\in \mathbb{R}^m$, $g(y) - \tfrac{1}{2} M_g^2 \mu \le g_{\mu}(y) \le g(y) $.
\end{enumerate}
\end{lemma}
\begin{proof}
Items (i), (ii) and (iii) follow from \cite[Lemmas~3.1, 3.2, 3.3 and 4.1]{BW21}.
Item (iv) follows from \cite[Lemma~5.4]{25ReLePa}.

Now, we consider item (v). Indeed, for any $\mu \in (0, \varrho^{-1})$ and $y\in \mathbb{R}^m$, 
    \begin{align*}
g(y) & \ge \inf_{w \in \mathbb{R}^m}\bigg\{g(w) + \frac{1}{2\mu} \| w - y \|^2\bigg\} = g_{\mu}(y) \notag
\\ & \ge \inf_{w \in \mathbb{R}^m}\bigg\{ g(y) - M_{g} \| w - y\| + \frac{1}{2\mu} \| w - y \|^2\bigg\}
= g(y) - \tfrac{1}{2} M_g^2 \mu, 
    \end{align*}
where the second inequality follows from the $M_g$-Lipschitz continuity of $g$.
\end{proof}

In view of Lemma~\ref{lm-wc-Moreau}(i), a smooth approximation of the objective in \eqref{opt-fgA} can be constructed as follows:
\begin{equation}\label{eq-def-Psi-mu}
\psi_{\mu}(x) \coloneqq f(x) + g_{\mu}(Ax) 
\quad \forall x\in \mathbb{R}^n, \mu\in (0, \varrho^{-1}),
\end{equation}
where $g_{\mu}$ is defined by \eqref{eq-def-gmu-ME}. We collect some properties of this smooth approximation in the next lemma.

\begin{lemma}\label{lm-psi-mu}
Consider \eqref{opt-fgA} and suppose that Assumption~\ref{ass-fgA} holds. 
Let $g_{\mu}$ and $\psi_{\mu}$ be defined by \eqref{eq-def-gmu-ME} and \eqref{eq-def-Psi-mu}, respectively.
Then, for any $\mu\in (0, \varrho^{-1})$, 
\begin{enumerate}[{\rm (i)}]
    \item $\psi_{\mu}$ is differentiable and has an $L_{\mu}$-Lipschitz continuous gradient $\nabla \psi_{\mu}$, where
\begin{equation}\label{eq-Lmu-vsm}
\nabla \psi_{\mu}(x) = \nabla f(x) + A^* \nabla g_{\mu}(Ax)
\quad \text{and} \quad
L_{\mu} \coloneqq L_f + \|A\|^2 \max \left\{ 
\mu^{-1}, \frac{\varrho}{1-\varrho \mu}
\right\};
\end{equation}
\item $\psi_{\mu}$ is $\varrho_{\mu}$-weakly convex,
where $\varrho_{\mu} \coloneqq L_f + \frac{\varrho}{1 - \varrho \mu}\|A\|^2$;

\item $\nabla \psi_{\mu}(x) \in \partial_{\nu_{\mu}}^{\varrho_{\mu}} \psi(x)$ for all $x\in \mathbb{R}^n$, where $\nu_{\mu} \coloneqq \frac{1}{2}M_g^2 \mu$.
\end{enumerate}
\end{lemma}

\begin{proof}
Item (i) follows from \cite[eq. (21)]{BW21}.
Item (ii) follows from Lemma~\ref{lm-wc-Moreau}(iv) together with an argument analogous to that in Remark~\ref{rm-fgA-wc}.

Finally, we prove item (iii). For any $x, z \in \mathbb{R}^n$,
\begin{align*}
    \psi(z) & \ge \psi_{\mu}(z) \ge \psi_{\mu}(x) + \langle \nabla \psi_{\mu}(x), z - x \rangle 
    - \frac{\varrho_{\mu}}{2} \| z - x \|^2
        \\ & \ge \psi(x) + \langle \nabla \psi_{\mu}(x), z - x \rangle 
          - \frac{ \varrho_{\mu}}{2} \| z - x \|^2 - \frac{1}{2} M_g ^2 \mu,
\end{align*}
where we use Lemma~\ref{lm-wc-Moreau}(v) to obtain the first and third inequalities, 
and use item (ii) to get the second inequality. 
This completes the proof.
\end{proof}

Algorithm~\ref{alg-vsm} below is adapted from the variable smoothing method proposed in \cite[Algorithm~1]{BW21}. The original algorithm employs the specific parameter choice of $r=1/3$, and its iteration complexity for finding an approximate stationary point was established in \cite[Theorem~4.1]{BW21}. Global sequential convergence has been recently established in \cite[Theorem~6.1]{26Tao} for $r\in (0.5,1]$ when $g$ is in addition convex based on the KL property of the function $\Psi_{q,\lambda}:\mathbb{R}^n\times \mathbb{R}\to (-\infty,\infty]$ defined as
\[
\Psi_{q,\lambda}(x,s) := \begin{cases}
\psi_{s^q}(x) + \lambda s^q & {\rm if}\ s > 0,\\
\psi(x)& {\rm if}\  s= 0,\\
\infty & {\rm otherwise},
\end{cases}
\]
for some $\lambda > 0$ and $q \ge 2$.

\begin{algorithm}[h]
\caption{Variable smoothing method for \eqref{opt-fgA}} 
\label{alg-vsm}
\begin{algorithmic}[0]
\setlength{\columnwidth}{\linewidth}
\State {\bf Require.} $r\in (0, 1)$, $x^{0}$, and the parameters $\varrho$ and $L_f$ from Assumption~\ref{ass-fgA}.
\For{$k=0,1,\ldots$} 
\State Set $\mu_k = (2\varrho)^{-1} (1+k)^{-r}$ and $\alpha_k = 1/L_{\mu_k}$, where $L_{\mu}$ is defined as in \eqref{eq-Lmu-vsm}.
\State Set $x^{k+1} = x^{k} - \alpha_{k} \nabla \psi_{\mu_k}(x^{k})$, 
where $\psi_{\mu}$ is defined in \eqref{eq-def-Psi-mu}.
\EndFor
\end{algorithmic}
\end{algorithm}

Here, we apply our convergence analysis framework to derive the global sequential convergence of $\{x^k\}$ when {\em $\psi$ is a KL function} under suitable choices of $r$, and discuss the {\em explicit local convergence rate} according to the KL exponent. We first present in the next lemma some basic properties of $\{x^k\}$ generated by Algorithm~\ref{alg-vsm}.
For the case $r=1/3$, Lemma~\ref{lm-vsm-Cs}(i) below was established in the proof of \cite[Theorem~4.1]{BW21}. The same argument extends readily to the case $r\in(0,1)$. For the sake of completeness and convenience of the readers, we include the details in the proof below.
\begin{lemma}\label{lm-vsm-Cs}
Consider \eqref{opt-fgA} and suppose that Assumption~\ref{ass-fgA} holds.
Let $r\in (0, 1)$ and let $\{x^{k}\}$ be generated by Algorithm~\ref{alg-vsm}.	
Then, the following statements hold.
\begin{enumerate}[{\rm (i)}]
\item We have for each $k\ge 0$,
\begin{equation}\label{des-fgA}
\psi_{\mu_{k+1}}(x^{k+1}) + M_{g}^2 \mu_{k+1} 
\le \psi_{\mu_{k}}(x^{k}) + M_{g}^2 \mu_{k}
- \frac{1}{2\alpha_k}  \| x^{k+1} - x^{k} \|^2.
\end{equation}

\item The sequence $\{x^k\} \subseteq \Omega_0$ is bounded.

\item For each $k\ge 0$, it holds that $\nabla \psi_{\mu_k}(x^{k+1}) \in \partial_{\nu_k}^{\widehat\varrho} \psi(x^{k+1})$, 
\begin{equation*}
\| \nabla \psi_{\mu_k}(x^{k+1})\| \le \frac{2}{\alpha_k} \| x^{k+1} - x^{k} \|
\quad \text{and} \quad \| \nabla \psi_{\mu_k}(x^{k+1})\| \le M_f + M_g \|A\|,
\end{equation*}
where $\widehat\varrho \coloneqq L_f + 2 \varrho \|A\|^2$, 
$M_f \coloneqq \sup\{ \| \nabla f(x) \| : x\in \Omega_0 \} < \infty$,
and $\nu_k \coloneqq \frac{1}{2}M_g^2 \mu_k$.
\end{enumerate}
\end{lemma}

\begin{proof}
First, we consider item (i). 
Let $k\ge 0$.
Since $\mu_k\le 1/(2\varrho) < 1/\varrho$, we see from Lemma~\ref{lm-psi-mu}(i) that, 
\begin{align*}
    \psi_{\mu_k}(x^{k+1}) & \le \psi_{\mu_k}(x^{k}) 
    + \langle \nabla \psi_{\mu_k}(x^{k}), x^{k+1} - x^k \rangle 
    + \frac{L_{\mu_k}}{2} \| x^{k+1} - x^k \|^2
   \\ & =  \psi_{\mu_k}(x^{k}) - \frac{1}{2\alpha_k} \| x^{k+1} - x^k \|^2,
\end{align*}
where the equality holds because $\nabla \psi_{\mu_k}(x^{k}) = - \alpha_k^{-1} ( x^{k+1} - x^{k})$ 
and $\alpha_k = 1/L_{\mu_k}$.
 
On the other hand, using \eqref{gmu12} and the fact $\mu_{k}/ \mu_{k+1} = (\frac{2+k}{1+k})^{r} \le 2^r \le 2$, we have 
\begin{equation*}
    \psi_{\mu_{k+1}}(x^{k+1}) \le \psi_{\mu_k}(x^{k+1}) + M_g^2 (\mu_{k} - \mu_{k+1}).
\end{equation*}
Summing the last two displays and rearranging terms gives \eqref{des-fgA}. 

Next, we consider item (ii). 
From Lemma~\ref{lm-wc-Moreau}(v), it follows that for all $k\ge 0$,
\begin{equation*}
    \psi(x^k) \le \psi_{\mu_{k}}(x^{k}) + \tfrac{1}{2}M_g^2 \mu_k 
    \le \psi_{\mu_{0}}(x^{0}) + M_{g}^2 \mu_{0}
    = \psi_{\mu_{0}}(x^{0}) + M_{g}^2 (2\varrho)^{-1},
\end{equation*}
where the second inequality holds by the nonincreasing property of 
$\{\psi_{\mu_{k}}(x^{k}) + M_g^2 \mu_k\}$ (see \eqref{des-fgA}).
This together with Assumption~\ref{ass-fgA}(iv)   
implies that $\{x^{k}\} \subseteq \Omega_0$ is bounded.

Finally, we prove item (iii). 
By Lemma~\ref{lm-psi-mu}(iii), 
we have $\nabla \psi_{\mu_k}(x^{k+1}) \in \partial_{\nu_k}^{\varrho_k} \psi(x^{k+1})$, 
where $\varrho_k \coloneqq L_f + \frac{\varrho}{1 - \varrho \mu_k} \|A\|^2$.
Since $\mu_k \le \mu_0 = (2\varrho)^{-1}$, we have $\varrho_k \le L_f + 2 \varrho \|A\|^2 = \widehat \varrho$,  which implies that $\nabla \psi_{\mu_k}(x^{k+1}) \in \partial_{\nu_k}^{\widehat \varrho} \psi(x^{k+1})$.

Moreover, by the triangle inequality,
\begin{align*}
    \| \nabla \psi_{\mu_k}(x^{k+1}) \| 
    & \le \| \nabla \psi_{\mu_k}(x^{k}) \| + \| \nabla \psi_{\mu_k}(x^{k+1}) - \nabla \psi_{\mu_k}(x^{k}) \| 
    \\ & \le \| \nabla \psi_{\mu_k}(x^{k}) \| + L_{\mu_k} \| x^{k+1} - x^{k} \| 
    = \frac{2}{\alpha_k} \| x^{k+1} - x^{k} \|,
\end{align*}
where the second inequality follows from Lemma~\ref{lm-psi-mu}(i).
Furthermore, 
\begin{align*}
    \| \nabla \psi_{\mu_k}(x^{k+1}) \| 
    & = \| \nabla f(x^{k+1}) + A^* \nabla g_{\mu_k}(Ax^{k+1})\|
    \\ & \le \| \nabla f(x^{k+1})\|  + \|A\| \|\nabla g_{\mu_k}(Ax^{k+1})\|
    \le M_f + M_g \|A\|,
\end{align*}
where the second inequality holds because of \eqref{nabl gmu} and the definition of $M_f$, and $M_f$ is finite thanks to item (ii).
\end{proof}

We are now ready to establish the convergence of Algorithm~\ref{alg-vsm} with $r\in (2/3,1)$ by showing that the $\{x^k\}$ generated satisfies (C1), (C2) and (C3) with suitable parameters, and invoking Theorem~\ref{th-main}. To the best of our knowledge, this is the first global sequential convergence and local convergence rate results for variable smoothing methods under some nonconvex setting based on the KL property of $\psi$ in \eqref{opt-fgA}. 

\begin{theorem}[Global convergence of Algorithm~\ref{alg-vsm}]
Consider \eqref{opt-fgA} and suppose that Assumption~\ref{ass-fgA} holds.
Let $r\in (\frac{2}{3}, 1)$ and let $\{x^{k}\}$ be generated by Algorithm~\ref{alg-vsm}.
Then, the following statements hold.
\begin{enumerate}[{\rm (i)}]
    \item If $\psi$ is a KL function, then $\{x^{k}\}$ converges to a critical point $x^*$ of $\psi$
and it holds that $\sum_{k=0}^{\infty}\|x^{k+1} - x^{k}\| < + \infty.$

\item If $\psi$ is a KL function with exponent $\theta \in (0, 1)$, then
$\| x^{k} - x^*\| = \mathcal{O}(k^{-\widetilde{r}})$,
where
\begin{equation*}
\widetilde{r} \coloneqq \widetilde{r}(r) \coloneqq \begin{cases}
\tfrac{3}{2}r - 1 
& \text{if \ } \theta \in (0, \tfrac{1}{2}] , 
\\ \min\left\{\tfrac{3}{2}r - 1, 
\frac{(1- \theta) r}{2 \theta}, 
\frac{(1 - \theta) (1 - r) }{2 \theta - 1} \right\}
& \text{if \ } \theta \in (\tfrac{1}{2}, 1).
\end{cases}
\end{equation*}
\end{enumerate}
\end{theorem}
\begin{proof}
Let $\widehat \varrho \coloneqq L_f + 2\varrho \|A\|^2$.
Then $\psi$ is $\widehat \varrho$-weakly convex by Remark~\ref{rm-fgA-wc}. 
Let $\hat{\psi}_k \coloneqq \psi_{\mu_k}(x^{k}) + M_g^2 \mu_k$, 
$\beta_k \coloneqq \frac{3}{2}M_g^2 \mu_k$, $\gamma_k \coloneqq 0$, 
$\nu_k \coloneqq \tfrac{1}{2}M_g^2\mu_k$ and $\xi^{k+1} \coloneqq \nabla \psi_{\mu_k}(x^{k+1})$ for all $k\ge 0$.
Now, by the definition of $\hat \psi^k$ and Lemma~\ref{lm-wc-Moreau}(v), 
we have 
\[
| \psi(x^k) - \hat{\psi}_k| \le \beta_k\ \ \ \forall k\ge 0.
\]
Then the sequence $\{x^k\}$ satisfies {\rm (C\ref{C1})} by the above display and \eqref{des-fgA}, 
satisfies {\rm (C\ref{C2})} by Lemma~\ref{lm-vsm-Cs}(iii), and satisfies {\rm (C\ref{C3})} by Lemma~\ref{lm-vsm-Cs}(ii) and the continuity of $\psi$.

For each $k\ge 0$, using the fact that $\mu_k \le \mu_0 = (2\varrho)^{-1}$, we deduce from the definition of $L_{\mu_k}$ in \eqref{eq-Lmu-vsm} that $L_{\mu_k} = L_f + \|A\|^2 \mu_k^{-1}$.
Since $\mu_k = \Theta(k^{-r})$, we obtain further that
$\alpha_k = 1/L_{\mu_k} = \Theta(k^{-r_{\alpha}})$,
$\beta_k = \mathcal{O}(k^{-r_{\beta}})$,
$\gamma_k = \mathcal{O}(k^{-r_{\gamma}})$ and 
$\nu_k = \mathcal{O}(k^{-r_{\nu}})$, 
with $r_{\alpha} = r_{\beta} = r_{\gamma} = r_{\nu} = r$.
Then 
\[
r_{\tau} = \min\left\{r_{\gamma}, \frac{\tau - 1}{\tau} r_{\beta}, \frac{1}{2} r_{\nu}\right\} = \frac{1}{2}r
\ {\rm and}\  r_{\alpha} + r_{\tau} - 1 = \frac{3}{2}r-1 > 0\ \ \ \forall \tau \ge 2.
\]
Therefore, the parameters $\{\alpha_k\}$, $\{\beta_{k}\}$, $\{\nu_k\}$ and $\{\gamma_{k}\}$ satisfy condition {\rm (A\ref{(B)})} of Theorem~\ref{th-main} with $\widetilde{\tau} =2$.
Then, item (i) follows from Theorem~\ref{th-main} directly, 
and item (ii) holds upon setting $\tau = \widetilde{\tau} = 2$ in Theorem~\ref{th-main}(iii) and (iv).
\end{proof}

\section{Proof of Theorem~\ref{th-main}}\label{sec4}
This section is devoted to proving Theorem~\ref{th-main}. We first establish in Proposition~\ref{prop-paras} some auxiliary properties of  the parameters $\{\alpha_k\}$, $\{\beta_{k}\}$, $\{\nu_k\}$ and $\{\gamma_{k}\}$ when either condition {\rm (A\ref{(A)})} or {\rm (A\ref{(B)})} in Theorem~\ref{th-main} holds. For ease of reference in our subsequent analysis, we collect some of these essential auxiliary properties as Property~\ref{ass-param}.

\begin{property}
\label{ass-param}
The following statements hold.
\begin{enumerate}[{\rm (i)}]
\item The positive sequence $\{\alpha_{k}\}$ satisfies $\sum_{l=0}^{\infty} \alpha_l = \infty$,
\begin{equation}\label{eq-alpha-bound}
\alpha_{k} \le \bar{\alpha} \quad \mbox{and} \quad 
\alpha_{k+1} \le M_{\alpha} \alpha_{k}
\quad \forall k\in \N_{0},
\end{equation}
for some constants $\bar{\alpha}>0$ and $M_{\alpha} > 0$.

\item The nonnegative sequences 
$\{\beta_{k}\}$, $\{\nu_k\}$ and $\{\gamma_{k}\}$ converge to $0$.			

\item There exists $\tau\ge 2$ such that
\begin{equation}\label{eq-sum-bdd}
\sum_{k=0}^{\infty} \alpha_{k} A_{k} < \infty
\quad \text{with} \quad
A_{l} \coloneqq  \gamma_{l}
+ \beta_{l+1}^{\frac{\tau-1}{\tau}} + \nu_{l}^{\frac{1}{2}}
\quad \forall l \in \N_{0}.
\end{equation}
\end{enumerate}
\end{property}

\begin{proposition}[Auxiliary properties of $\{\alpha_k\}$, $\{\beta_{k}\}$, $\{\nu_k\}$ and $\{\gamma_{k}\}$]\label{prop-paras}
	Let $\{\alpha_{k}\}$ be a sequence of positive numbers.
	Let $\{\beta_{k}\}$, $\{\nu_k\}$ and $\{\gamma_{k}\}$ be sequences of nonnegative numbers.
    Suppose that either Theorem~\ref{th-main}{\rm (A\ref{(A)})} or {\rm (A\ref{(B)})} holds.
	Then, Property~\ref{ass-param} holds.
    Moreover, the following statements hold.
	\begin{enumerate}
		[{\rm (i)}]
		\item If Theorem~\ref{th-main}{\rm (A\ref{(A)})} holds, 
        then there exists $q\in (0,1)$ such that for any $\tau \ge 2$, 
        the following estimates hold:
		$$
		A_k = \mathcal{O}(q^{k})\quad \text{and} \quad\sum_{l = k}^{\infty} \alpha_{l-1} A_{l-1} = \mathcal{O}(q^{k}).
		$$		
		\item If Theorem~\ref{th-main}{\rm (A\ref{(B)})} holds, 
        then for any $\tau\in [\widetilde{\tau}, \infty)$, 
        the following estimates hold: 
        $$
		A_{k} = \mathcal{O}(k^{-r_{\tau}}) 
		\quad \text{and} \quad
		\sum_{l = k}^{\infty} \alpha_{l-1} A_{l-1} = \mathcal{O}(k^{-(r_{\alpha} + r_{\tau} - 1)}).
		$$ 
	\end{enumerate}
\end{proposition}
\begin{proof}
	In both settings, Property~\ref{ass-param}(ii) holds trivially under either condition {\rm (A\ref{(A)})} or {\rm (A\ref{(B)})} in Theorem~\ref{th-main}. 
	Thus, we only need to show both Property~\ref{ass-param}(i) and (iii) hold under {\rm (A\ref{(A)})} or {\rm (A\ref{(B)})} in Theorem~\ref{th-main}. We divide our discussion into two cases.
	
	\underline{Case 1}: Suppose that Theorem~\ref{th-main}{\rm (A\ref{(A)})} holds.
    Since $\alpha_{k} = \Theta(1)$ and $\{\alpha_k\}\subset\R_{++}$, there exist $0<\underline{\alpha}<\bar{\alpha}$ such that $\underline{\alpha}\le \alpha_{k} \le \bar{\alpha}$ for all $k\ge 0$. 
	Then $\alpha_{k+1} \le \bar{\alpha} = M_{\alpha} \underline{\alpha}  \le M_{\alpha} \alpha_{k}$ for all $k\ge 0$, 
	where $M_{\alpha} \coloneqq {\bar \alpha}/\underline{\alpha}$.
	Clearly, $\sum_{k=0}^{\infty} \alpha_{k} = \infty$.
	This shows that Property~\ref{ass-param}(i) holds.
	
	Next, for any $\tau \ge 2$, by definition, we have
	\begin{equation*}
		A_{k} = \gamma_k + \beta_{k+1}^{\frac{\tau - 1}{\tau}} + \nu_{k}^{\frac{1}{2}}
		=  \mathcal{O}(q_{\gamma}^{k} + q_{\beta}^{\frac{\tau - 1}{\tau}k} + q_{\nu}^{\frac{k}{2}} )
        = \mathcal{O}(q_{\gamma}^{k} + q_{\beta}^{\frac{k}{2}} + q_{\nu}^{\frac{k}{2}} )
		= \mathcal{O}(q_{A}^{k}),
	\end{equation*}
	where $q_{A} \coloneqq \max\{q_{\gamma}, q_{\beta}^{\frac12}, q_{\nu}^{\frac{1}{2}} \} \in (0, 1)$.
	Thus, there exists $M_1>0$ such that $A_{k} \le M_1 q_{A}^{k}$ for all $k\ge 0$.
	This further implies that $\sum_{k=0}^{\infty} \alpha_{k} A_{k} \le \bar{\alpha} M_1 \sum_{k=0}^{\infty} q_A^{k} < \infty$ and 
	\begin{equation*}
		\sum_{l = k}^{\infty} \alpha_{l-1} A_{l-1} \le 
		\bar{\alpha} M_1 \sum_{l = k}^{\infty} q_{A}^{l-1} 
		=  \frac{\bar{\alpha} M_1 q_{A}^{k-1}}{1 - q_{A}}.
	\end{equation*}

	\underline{Case 2}: Suppose that Theorem~\ref{th-main}{\rm (A\ref{(B)})} holds. 
    Then there exist $M_5>M_4>0$ such that $M_4(1+k)^{-r_{\alpha}} \le \alpha_k \le M_5 (1+k)^{-r_{\alpha}}$ for all $k\ge 0$.
	Thus, we have, for all $k\ge 0$,
	$$\alpha_{k+1} \le M_5 (2+k)^{-r_{\alpha}} \le M_5 (1+k)^{-r_{\alpha}} \le (M_5/M_4) \alpha_{k} \eqqcolon M_{\alpha} \alpha_{k}.
	$$
	Moreover, $\alpha_{k} \le M_5 (1+k)^{-r_{\alpha}}  \le M_5 \eqqcolon \bar{\alpha}$ for all $k\ge 0$.
	Next, since $r_{\alpha} \in [0, 1)$, 
    we have $\sum_{k=0}^{\infty}(1+k)^{-r_{\alpha}} = \infty$, which implies that $\sum_{k=0}^{\infty}\alpha_k = \infty.$
	Therefore, Property~\ref{ass-param}(i) holds.
	
	Finally, let $\tau\in [\widetilde{\tau}, \infty)$. By definition, we have  
	\begin{equation*}
		A_{k} = \gamma_{k}
		+ \beta_{k+1}^{\frac{\tau-1}{\tau}} + \nu_{k}^{\frac{1}{2}}
		= \mathcal{O}( k^{-r_{\gamma}} + k^{-\frac{\tau-1}{\tau}r_{\beta}} + k^{-\frac{1}{2}r_{\nu}}) 
		= \mathcal{O}(k^{-r_{\tau}}),
	\end{equation*}
	where the last equality follows from the definition of $r_{\tau}$.
	This implies that $\alpha_{k}A_{k} = \mathcal{O}(k^{-(r_{\tau}+r_{\alpha})})$.
	Since $r_{\alpha} + r_{\tau} \ge r_{\alpha} + r_{\widetilde\tau} >1$, 
    we conclude that $\sum_{k=0}^{\infty}\alpha_k A_k < \infty$ and $\sum_{l = k}^{\infty} \alpha_{l-1} A_{l-1}  = \mathcal{O}(k^{-(r_{\alpha} + r_{\tau} - 1)})$.	
\end{proof}

The next lemma concerns the set of accumulation points of $\{x^k\}$.

\begin{lemma}[Cluster points of $\{x^k\}$]\label{lm-basic}
Let $\psi:\mathbb{X}\to (-\infty, \infty]$ be a proper, closed, 
and $\varrho$-weakly convex function with constant $\varrho\ge 0$. 
Suppose that Property~\ref{ass-param}(i) and (ii) hold.
Let $\{x^{k}\}$ satisfy {\rm (C\ref{C1})}, {\rm (C\ref{C2})} and {\rm (C\ref{C3})}.
Let $\overline{\Omega}$ be the set of accumulation points of $\{x^k\}$. 
Then the following statements hold.
\begin{enumerate}[{\rm (i)}]
\item The set $\overline{\Omega}$ is nonempty and compact, 
and $\psi$ is finite and constant on $\overline{\Omega}$.
Moreover, 
\begin{equation}\label{eq-value-lim}
\overline{\psi} \coloneqq \lim_{k\to \infty} \hat{\psi}_{k} =
\lim_{k\to \infty}\psi(x^k) 
=  \psi(\bar{x}) \in \mathbb{R} \quad \forall \bar{x}\in \overline{\Omega}.
\end{equation}

\item  $\lim_{k\to\infty} \|x^{k+1} - x^{k}\| = \lim_{k\to\infty}\widetilde{\nu}_{k} = 0$, 
where $\{\widetilde{\nu}_k\}$ is given in \eqref{eq-nu-tilde}.

\item $\liminf_{k\to \infty} \|\xi^{k+1}\| = 0$.

\item  $\overline{\Omega}\subseteq \dom \partial \psi$, and $\overline{\Omega}$ contains a critical point of $\psi$.
\end{enumerate}
\end{lemma}
\begin{proof}
(i) Because $\{x^k\}$ is bounded (see {\rm (C\ref{C3})}), we have that $\overline{\Omega}$ is nonempty and compact.
Let $\bar{x} \in \overline{\Omega}$.
Then there exists a subsequence $\{x^{k_j}\}$ such that $\lim_{j\to \infty}x^{k_j} = \bar{x}$ and it holds that $\psi(\bar x)\in \mathbb{R}$ by (C3). 
Moreover, we have
\begin{align*}
& |\hat{\psi}_{k_j} - \psi(\bar{x})| 
\le |\hat{\psi}_{k_j} - \psi(x^{k_j})| + |\psi(x^{k_j}) - \psi({\bar x})| 
\\ &\le \beta_{k_j} + |\psi(x^{k_j}) - \psi({\bar x})| \to 0 
\quad \mbox{as}\ \ j \to \infty,
\end{align*}
where the second inequality holds thanks to \eqref{eq-f-psi}, 
and the limit holds because of \eqref{eq-continuity} and the fact that $\lim_{k\to\infty} \beta_k = 0$ (see Property~\ref{ass-param}(ii)).
Since $\{\hat{\psi}_{k}\}$ is nonincreasing (see \eqref{eq-f-beta-des}), we know that $\lim_{k\to\infty}\hat{\psi}_k$ exists and
the last display implies that 
$\psi(\bar{x}) = \lim_{k\to\infty}\hat{\psi}_k$.
This, together with \eqref{eq-f-psi}, proves \eqref{eq-value-lim}.

(ii) Let $k\in \mathbb{N}_0$. 
Using \eqref{eq-f-beta-des}, we have
\begin{align*}
\frac{1}{\alpha_{k}}\|x^{k+1} - x^{k} \|^2 
& \le 
\frac{\hat{\psi}_k - \hat{\psi}_{k+1}}{a_1}
\le 
\frac{\hat{\psi}_k - \overline{\psi}}{a_1}
\to 0 \quad \mbox{as}\ \ k \rightarrow \infty,
\end{align*}
where we have used the monotonicity of $\{\hat{\psi}_k\}$ (see \eqref{eq-f-beta-des}) and \eqref{eq-value-lim} for the second inequality. 
This, together with \eqref{eq-alpha-bound}, 
implies that $\lim_{k\to \infty}\|x^{k+1} - x^{k} \| = 0$.
Furthermore, using \eqref{eq-value-lim}, $\lim_{k\to \infty}\nu_k = 0$  (see Property~\ref{ass-param}(ii)) 
and the definition of $\{\widetilde\nu_k\}$, 
we obtain that $\lim_{k\to \infty}\widetilde{\nu}_k = 0$.

(iii) Suppose to the contrary that there exist
$\varepsilon>0$ and  $N_1 \in \N_{0}$ such that
$\|\xi^{k+1}\| \ge \varepsilon$ for all $k \ge N_1$.
Since $\lim_{k\to\infty} \gamma_{k} = 0$ thanks to Property~\ref{ass-param}(ii), there exists 
$N_2 \in \N_{0}$ such that
$\gamma_{k} \le \varepsilon/2$ for all $k \ge N_2$.
It then follows from \eqref{eq-C2} that
\begin{equation*}
\frac{a_{3}}{\alpha_{k}} \|x^{k+1} - x^{k}\| \ge  \|\xi^{k+1}\| - \gamma_{k} \ge \frac{\varepsilon}{2} \quad \forall k \ge N \coloneqq \max\{N_1, N_2\}.
\end{equation*}
This yields that
\begin{equation*}
\frac{a_1\varepsilon^2}{4} \alpha_{k}
\le \frac{a_{1}a_{3}^2}{\alpha_k} \| x^{k+1} - x^{k} \|^2
\le a_{3}^2(\hat{\psi}_k - \hat{\psi}_{k+1}) \quad \forall k\ge N,
\end{equation*}
where the last inequality holds thanks to \eqref{eq-f-beta-des}.
Summing the last display from $k = N$ to $\infty$, and invoking \eqref{eq-value-lim}, we obtain
\begin{equation*}
\frac{a_1\varepsilon^2}{4} \sum_{k=N}^{\infty}\alpha_{k}
\le a_{3}^2(\hat{\psi}_N - \overline\psi)
< \infty,
\end{equation*}
which contradicts $\sum_{k=0}^{\infty} \alpha_k = \infty$ 
(see Property~\ref{ass-param}(i)).

(iv) Notice from {\rm (C\ref{C2})} that
$\xi^{k+1} \in \partial^{\varrho}_{\widetilde{\nu}_{k}} \psi (x^{k +1})$ for all $k\ge 0$ and $\{\xi^{k+1}\}$ is bounded. 

We first show that $\overline{\Omega}\subseteq \dom \partial\psi$. Let $\bar x\in \overline{\Omega}$. Then there exists a subsequence $\{x^{k_j}\}$ such that $\lim_{j\to\infty} x^{k_j+1} = \bar{x}$. Using this and {\rm (C\ref{C3})}, we also have
$\lim_{j\to \infty} \psi(x^{k_j+1}) = \psi(\bar{x}) \in \R$. Next, since $\{\xi^{k+1}\}$ is bounded, by passing to a further subsequence if necessary, we may assume that $\lim_{j\to\infty} \xi^{k_j+1} = \bar{\xi}$ for some $\bar{\xi}$. Since we have $\lim_{k\to \infty}\widetilde{\nu}_{k} = 0$ from item (ii), we can now deduce upon passing to the limit in $\xi^{k_j+1} \in \partial^{\varrho}_{\widetilde{\nu}_{k_j}} \psi (x^{k_j+1})$ that $\bar \xi\in \partial \psi(\bar x)$. In particular, $\bar x\in \dom \partial\psi$. 

Finally, we argue that $\overline{\Omega}$ contains a critical point of $\psi$. To this end, note that $\liminf_{k\to \infty} \|\xi^{k+1}\| = 0$ (see item (iii)), meaning that
there exists a subsequence $\{\xi^{k_j}\}$ such that 
$\lim_{j\to \infty}  \|\xi^{k_j+1}\| = 0$.
Since $\{x^{k}\}$ is bounded (see {\rm (C\ref{C3})}), 
by passing to a further subsequence if necessary, we may assume that 
$\lim_{j\to\infty} x^{k_j+1} = \bar{x}$ for some $\bar x \in \overline{\Omega}$.
Then, from {\rm (C\ref{C3})}, we have
$\lim_{j\to \infty} \psi(x^{k_j+1}) = \psi(\bar{x}) \in \R$.
Combining these facts, we conclude upon passing to the limit in $\xi^{k_j+1} \in \partial^{\varrho}_{\widetilde{\nu}_{k_j}} \psi (x^{k_j+1})$ that $0 \in \partial \psi(\bar x)$.
\end{proof}

\begin{lemma}[Finite length property]\label{lm-kl-x}
Let $\psi:\mathbb{X}\to (-\infty, \infty]$ be a proper, closed, 
and $\varrho$-weakly convex function with constant $\varrho\ge 0$. 
Suppose that Property~\ref{ass-param} holds.
Let $\{x^{k}\}$ be a sequence that satisfies {\rm (C\ref{C1})}, {\rm (C\ref{C2})}, and {\rm (C\ref{C3})}, and let $\overline{\Omega}$ denote the set of accumulation points of $\{x^k\}$. Suppose that $\psi$ is a KL function, and that $\hat{\psi}_k > \overline{\psi}$ for all $k\ge 0$, 
where $\overline{\psi}$ is given in \eqref{eq-value-lim}.
Then there exist $\eta>0$, $\varphi \in {\mathcal{Q}}_{\eta}$,
$N\in \mathbb{N}$, $c_1>0$ and $c_2>0$ such that, for any $k\ge N$,
\begin{equation}
h_{k} \coloneqq \hat{\psi}_k - \overline{\psi} < \eta 
\quad \text{and} \quad 
\frac{c_1}{\alpha_{k-1}} \left(
\Delta_{k-1}
+ \alpha_{k-1}A_{k-1}
\right)  \varphi^{\prime}(h_k) \ge 1
\label{eq-KL-ineq-2}
\end{equation}
and
\begin{align}\label{eq-x-sig}
\Delta_k
& \le \sigma_k - \sigma_{k+1},
\end{align}
where
\begin{align}\label{def-Delta}
& \Delta_l \coloneqq \|x^{l+1} - x^{l}\| + \alpha_{l} (\hat{\psi}_{l} - \hat{\psi}_{l+1})^{\frac{1}{2}},
\\& \sigma_l \coloneqq \Delta_{l-1} + c_2 \, \varphi(h_l) + \sum_{j=l}^{\infty}\alpha_{j-1}A_{j-1} < \infty
\quad \forall l\ge N.\label{eq-sig}
\end{align}
Moreover, 
\begin{equation}
	\label{eq-sum-sig}
\sum_{l=k}^{\infty}\Delta_l \le \sigma_k \quad \forall k\ge N.	
\end{equation}
\end{lemma}

\begin{proof}		
Let $\Psi_{\tau}$ be defined by \eqref{eq-F-tau-def}, with $\tau \ge 2$ given in Property~\ref{ass-param}(iii).	
From Lemma~\ref{lm-basic}(i), it follows that $\overline{\Omega}\times \{0\}$ is nonempty and compact,
and $\Psi_{\tau}(x, t) = \overline{\psi} \in \mathbb{R}$ for all $(x, t) \in \overline{\Omega}\times \{0\}$. Since $\overline{\Omega}\subseteq \dom \partial \psi$ according to Lemma~\ref{lm-basic}(iv) and $\psi$ is a KL function by assumption, we deduce from
 Corollary~\ref{cor-kl} that $\overline{\Omega}\times \{0\}\subseteq \dom \partial\Psi_{\tau}$ and $\Psi_{\tau}$ satisfies the KL property at each point of 
$\overline{\Omega}\times \{0\}$.
Using Lemma~\ref{lm-KL-unif}, there exist $\eta>0$, $\delta\in (0, 1)$,
and $\varphi \in {\mathcal{Q}}_{\eta}$ such that,
for all $(x, t) \in \mathbb{X}\times (-\delta, \delta)$
with $\mathrm{dist}(x, \overline{\Omega}) < \delta$ and $\overline{\psi}< \psi(x) + |t|^{\tau} < \overline{\psi}+\eta$,
one has
\begin{equation}\label{eq-kl-Psi}
\varphi^{\prime}( \psi(x) + |t|^{\tau} - \overline{\psi})
\mathrm{dist}(0, \partial \Psi_{\tau}(x, t)) \ge 1.
\end{equation}

Now, for each $k\ge 1$, by Lemma~\ref{lm-BR} and the fact that $\xi^{k} \in \partial_{\widetilde{\nu}_{k-1}}^{\varrho} \psi(x^{k})$ (see {\rm (C\ref{C2})}),
there exist $z^{{k}} \in \mathbb{X}$ and $\zeta^{{k}} \in \partial \psi(z^{{k}})$ such that
\begin{align}\label{eq-x-z-xi-zeta}
& \|z^{{k}} - x^{{k}}\| \le \widetilde{\nu}_{k-1}^{\frac{1}{2}},
\quad
\|\xi^{k} -\zeta^{{k}}\| \le (1+\varrho)\widetilde{\nu}_{k-1}^{\frac{1}{2}},
\\\label{eq-fxz}
& |\psi(x^{{k}}) - \psi(z^{{k}})|
\le \widetilde{\nu}_{k-1}^{\frac{1}{2}} \|\xi^{{k}}\| + (1 + \tfrac{\varrho}{2})\widetilde{\nu}_{k-1}.
\end{align}
Moreover, by the triangle inequality,
\begin{align}
\notag & | \hat{\psi}_k - \psi(z^{k}) |
\le |\hat{\psi}_k - \psi(x^k)| + |\psi(x^{k}) - \psi(z^{k})|
\\& \le \widetilde{\nu}_{k-1}^{\frac{1}{2}} \|\xi^{k}\| + (1+\tfrac{\varrho}{2})\widetilde{\nu}_{k-1} + \beta_{k}
\eqcolon s_k \quad \forall k\in \mathbb{N},
\label{eq-s-def}
\end{align}
where the second inequality holds by \eqref{eq-f-psi} and \eqref{eq-fxz}.
Consequently,
\begin{align}\label{eq-hk-fzs}
0 & \le \hat{\psi}_k - \overline{\psi}  \le
\psi(z^{k}) + s_k - \overline{\psi}
\le \hat{\psi}_k - \overline{\psi} + 2s_k \quad \forall k\in \mathbb{N},
\end{align}
where the nonnegativity holds because $\{\hat{\psi}_k\}$ is nonincreasing (see \eqref{eq-f-beta-des})
and $\lim_{k\to\infty} \hat{\psi}_k = \overline{\psi}$ (see \eqref{eq-value-lim}),
and the last two inequalities hold by \eqref{eq-s-def}.

Since $\{\xi^{k+1}\}$ is bounded (see \eqref{eq-C2}), it follows from Lemma~\ref{lm-basic}(ii)
and Property~\ref{ass-param}(ii) that $\lim_{k\to \infty}s_{k} = 0$.
Furthermore,
\begin{equation*}
\mathrm{dist}(z^{k}, \overline{\Omega})
\le \mathrm{dist}(x^{k}, \overline{\Omega})
+ \| z^{k} - x^{k} \|
\le \mathrm{dist}(x^{k}, \overline{\Omega}) + \widetilde{\nu}_{k-1}^{\frac{1}{2}}
\to 0 \quad \text{as} \quad k \to \infty,
\end{equation*}
where the second inequality holds by \eqref{eq-x-z-xi-zeta},
and the limit holds by Lemma~\ref{lm-basic}(ii) and the definition of $\overline{\Omega}$.
Combining these facts with \eqref{eq-hk-fzs} and \eqref{eq-value-lim},
there exists $N \in \mathbb{N}$ such that, for all $k\ge N$,
\begin{align}\label{eq-KL-region}
&\mathrm{dist}(z^{k}, \overline{\Omega}) < \delta,
\quad 0 \le s_k < \delta^{\tau}<1,
\quad 0 \le \hat{\psi}_k - \overline{\psi} \le \psi(z^{k}) + s_{k} - \overline{\psi} < \eta.
\end{align}

Next, recall by assumption that $\hat{\psi}_k > \overline{\psi}$ for all $k\ge 0$. Then, for any $k\ge N$, one can set $(x, t) = (z^{k}, s_k^{\frac{1}{\tau}})$ in \eqref{eq-kl-Psi}, which gives
\begin{equation} \label{eq-kl-z}
1 \le \varphi^{\prime}( \psi(z^{k}) + s_k - \overline{\psi} )
\mathrm{dist}(0, \partial \Psi_{\tau}(z^{k}, s_k^{\frac{1}{\tau}}))
\le \varphi^{\prime}( \hat{\psi}_{k} - \overline{\psi} )
\mathrm{dist}(0, \partial \Psi_{\tau}(z^{k}, s_k^{\frac{1}{\tau}})),
\end{equation}
where the second inequality follows from the monotonicity of $\varphi^{\prime}$ on $(0, \eta)$ and the third relation in \eqref{eq-KL-region}.

Since $(\xi, r) \in \partial \Psi_{\tau}(z^{k}, s_k^{\frac{1}{\tau}})$ if and only if
$\xi \in \partial \psi(z^{k})$ and $r = \tau s_{k}^{\frac{\tau - 1}{\tau}}$, we have
\begin{align}
\notag & \mathrm{dist}(0, \partial \Psi_{\tau}(z^{k}, s_{k}^{\frac{1}{\tau}}))
= \sqrt{\mathrm{dist}(0, \partial \psi(z^{k}))^2 + (\tau s_k^{\frac{\tau-1}{\tau}})^2}
\le  \mathrm{dist}(0, \partial \psi(z^{k})) + \tau s_{k}^{\frac{\tau-1}{\tau}}
\\ & \le \|\zeta^{k}\| + \tau s_{k}^{\frac{\tau-1}{\tau}}
\le \|\xi^{k}\| + \| \zeta^{k} - \xi^{k}\| + \tau s_{k}^{\frac{\tau-1}{\tau}}
\le \|\xi^{k}\| + (1 + \varrho)\widetilde{\nu}_{k-1}^{\frac{1}{2}} + \tau s_{k}^{\frac{\tau-1}{\tau}},
\label{eq-part-Psi-upbd}
\end{align}
where the second inequality holds because $\zeta^{k} \in \partial \psi(z^{k})$,
and the last inequality follows from \eqref{eq-x-z-xi-zeta}.
Moreover, it follows from the definition of $s_k$ in \eqref{eq-s-def} that
\begin{align}
 \tau s_k^{\frac{\tau-1}{\tau}}
 & = \tau  \left( \widetilde{\nu}_{k-1}^{\frac{1}{2}}\|\xi^{k}\|  + (1 + \tfrac{\varrho}{2})\widetilde{\nu}_{k-1} +  \beta_k\right)^{\frac{\tau-1}{\tau}}\notag
 \\ & \le \tau  \left( (\widetilde{\nu}_{k-1}^{\frac{1}{2}}\|\xi^{k}\|)^{\frac{\tau-1}{\tau}}
+ (\tfrac{2+\varrho}{2}\widetilde{\nu}_{k-1})^{\frac{\tau-1}{\tau}} +  \beta_k^{\frac{\tau-1}{\tau}}\right) \notag
\\ & \notag \le \tau 
\left( \frac{\tau - 1}{\tau} \|\xi^{k}\| 
+ \frac{1}{\tau} \widetilde{\nu}_{k-1}^{\frac{\tau - 1}{2}}
+ (\tfrac{2+\varrho}{2}\widetilde{\nu}_{k-1})^{\frac{\tau-1}{\tau}} +  \beta_k^{\frac{\tau-1}{\tau}} \right)
\le \widetilde{c}_1
( \|\xi^{k}\| 
+ \widetilde{\nu}_{k-1}^{\frac{1}{2}} +  \beta_k^{\frac{\tau-1}{\tau}} ),
\end{align}
where the first inequality holds because $\frac{\tau - 1}{\tau}<1$, 
the second inequality holds by applying
Young's inequality $t_1 t_2 \le \frac{\tau - 1}{\tau} t_1^{\frac{\tau}{\tau - 1}} + \frac{1}{\tau}t_2^{\tau}$ with $t_1 = \|\xi^k\|^{\frac{\tau - 1}{\tau}}$ and 
$t_2 = \widetilde{\nu}_{k-1}^{\frac{\tau - 1}{2 \tau}}$, 
and the last inequality holds due to $\widetilde{\nu}_{k-1} \le s_k < 1$ (see \eqref{eq-s-def} and \eqref{eq-KL-region}), $\tau \ge 2$, and upon defining 
$$
\widetilde{c}_1 \coloneqq 
\tau \max \bigg\{\frac{\tau-1}{\tau}, \frac{1}{\tau} + \bigg(\frac{2+\varrho}{2}\bigg)^{\frac{\tau-1}{\tau}}, 1\bigg\}.
$$
This further yields
\begin{align}
& \|\xi^{k}\|
+ (1+\varrho) \widetilde{\nu}_{k-1}^{\frac{1}{2}}
+ \tau s_k^{\frac{\tau-1}{\tau}} \notag
\le (1 + \widetilde{c}_1) \|\xi^{k}\|
+ (1+\varrho + \widetilde{c}_1)
\widetilde{\nu}_{k-1}^{\frac{1}{2}} 
+ \widetilde{c}_1  \beta_k^{\frac{\tau-1}{\tau}}
\notag
\\ & \overset{\mathrm{(a)}}{\le} (1 + \widetilde{c}_1) \left(\frac{a_{3}}{\alpha_{k-1}} \|x^{k} - x^{k-1}\| + \gamma_{k-1} \right)
+ \widetilde{c}_2 \widetilde\nu_{k-1}^{\frac{1}{2}} + \widetilde{c}_1 \beta_k^{\frac{\tau-1}{\tau}} \notag
\\ & \notag \overset{\mathrm{(b)}}{\le}  
(1 + \widetilde{c}_1) \left(\frac{a_{3}}{\alpha_{k-1}} \|x^{k} - x^{k-1}\| + \gamma_{k-1} \right)
+ \widetilde{c}_2 \left( \sqrt{a_2 (\hat{\psi}_{k-1} - \hat{\psi}_{k} ) }  + \nu_{k-1}^{\frac{1}{2}} \right) 
+ \widetilde{c}_1 \beta_k^{\frac{\tau-1}{\tau}}
\\ \notag & \overset{\mathrm{(c)}}{\le}  
 \frac{c_1}{\alpha_{k-1}} \left(
\|x^{k} - x^{k-1}\| + \alpha_{k-1} (\hat{\psi}_{k-1} - \hat{\psi}_{k})^{\frac{1}{2}} + \alpha_{k-1}(\gamma_{k-1}  + \beta_k^{\frac{\tau-1}{\tau}} + \nu_{k-1}^{\frac{1}{2}})
\right),
\end{align}
where $\widetilde{c}_2$ and $c_1$ are given by
\begin{equation*}
\widetilde{c}_2 \coloneqq 1+\varrho + \widetilde{c}_1, 
\quad
c_1\coloneqq \max\{a_{3}(1+\widetilde{c}_1), \widetilde{c}_2 \sqrt{a_2},
1 + \varrho + \widetilde{c}_1\},
\end{equation*}
(a) follows from the first relation in \eqref{eq-C2},
(b) follows from the definition of $\widetilde{\nu}_{k-1}$ in \eqref{eq-nu-tilde} and the inequality $\sqrt{t_1 + t_2} \le \sqrt{t_1} + \sqrt{t_2}$ for all $t_1, t_2 \ge 0$,
(c) holds upon invoking the definitions of $c_1$ and $\widetilde{c}_2$.
This, combined with \eqref{eq-KL-region}, \eqref{eq-kl-z} and \eqref{eq-part-Psi-upbd}, implies \eqref{eq-KL-ineq-2}.

Next, we prove \eqref{eq-x-sig}.
Let $k\ge N$. By \eqref{eq-f-beta-des} and $\alpha_k\le \bar \alpha$ (see \eqref{eq-alpha-bound}), we have 
\begin{align}
    \hat{\psi}_{k} - \hat{\psi}_{k+1}
    & \ge \frac{a_1}{2\alpha_k} \| x^{k+1} - x^{k}\|^2
+ \frac{\alpha_k}{2\bar{\alpha}} (\hat{\psi}_{k} - \hat{\psi}_{k+1} ) \notag
\\ & \ge \frac{\min\{a_1, \bar\alpha^{-1}\}}{2\alpha_k} ( \| x^{k+1} - x^{k}\|^2 + \alpha_{k}^2(\hat{\psi}_{k} - \hat{\psi}_{k+1}) )
 \ge \frac{\min\{a_1, \bar\alpha^{-1}\}}{4\alpha_k} \Delta_{k}^2,
 \label{dpsi-lowb}
\end{align}
where the third inequality holds because of the definition of $\Delta_k$ in \eqref{def-Delta}
and the fact that $a^2 + b^2 \ge \frac{1}{2}(a+b)^2$ for all $a,b\ge 0$.
Since $\varphi$ is concave on $(0, \eta)$, we have
\begin{align*}
\varphi(h_{k}) - \varphi(h_{k+1}) 
& \ge \varphi^\prime(h_{k}) (\hat{\psi}_{k} - \hat{\psi}_{k+1} ) 
\ge \frac{\min\{a_1, \bar\alpha^{-1}\} \alpha_{k-1} \Delta_k^2}
{4c_1 \alpha_{k}(\Delta_{k-1} + \alpha_{k-1}A_{k-1})},
\end{align*}
where the second inequality holds thanks to \eqref{dpsi-lowb} and \eqref{eq-KL-ineq-2}.
Rearranging terms in the above display yields
\begin{align*}
\Delta_k^2 & \le \frac{4c_1\alpha_{k}}{\min\{a_1, \bar\alpha^{-1}\} \alpha_{k-1}}
(\varphi(h_{k}) - \varphi(h_{k+1})) (\Delta_{k-1} + \alpha_{k-1}A_{k-1})
\\ & \le c_2 (\varphi(h_{k}) - \varphi(h_{k+1})) (\Delta_{k-1} + \alpha_{k-1}A_{k-1}),
\end{align*}
where the second inequality holds upon invoking \eqref{eq-alpha-bound} and defining $c_2 \coloneqq 4c_1 M_{\alpha} / \min\{a_1, \bar{\alpha}^{-1} \}$.

Taking square roots 
 on both sides of the last display and using the inequality 
$\sqrt{t_1 t_2} \le (t_1 + t_2)/2 $ for all $t_1,t_2\ge 0$, 
we obtain that
\begin{align*}
2 \Delta_k
\le \Delta_{k-1} + c_2 (\varphi(h_{k}) - \varphi(h_{k+1})) + 	\alpha_{k-1}A_{k-1}.
\end{align*}
Moreover, it follows from \eqref{eq-sum-bdd} that $\sigma_l$ in \eqref{eq-sig} is finite, for each $l\ge N$. 
This, combined with the last display, implies \eqref{eq-x-sig}.

Finally, for any $k_2\ge k_1\ge N$,
summing \eqref{eq-x-sig} from $k = k_1 $ to $k_2$,
we have 
\begin{equation*}
	\sum_{k=k_1}^{k_2} \Delta_k 
	\le \sigma_{k_1} - \sigma_{k_2+1} \le \sigma_{k_1},
\end{equation*}
where the last inequality holds thanks to the nonnegativity of $\{\sigma_l\}$.
Fixing $k_1$ and passing to the limit as $k_2\to \infty$ in the last display, we conclude that \eqref{eq-sum-sig} holds.
\end{proof}

\begin{theorem}[Global convergence under KL property]\label{th-global-convergence}
Let $\psi:\mathbb{X}\to (-\infty, \infty]$ be a proper, closed, 
and $\varrho$-weakly convex function with constant $\varrho\ge 0$. 
Suppose that Property~\ref{ass-param} holds.
Let $\{x^{k}\}$ be a sequence that satisfies {\rm (C\ref{C1})}, {\rm (C\ref{C2})}, and {\rm (C\ref{C3})}.
Suppose that $\psi$ is a KL function.
Then $\{x^{k}\}$ converges to a critical point of $\psi$.
Moreover, it holds that $\sum_{k=0}^{\infty}\|x^{k+1} - x^{k}\| < + \infty$.
\end{theorem}
\begin{proof}
If $\hat{\psi}_{k_1} = \overline{\psi}$ for some $k_1$, 
then by \eqref{eq-f-beta-des} and \eqref{eq-value-lim}, 
$\hat{\psi}_k = \overline{\psi}$ and $x^{k+1} = x^{k}$ for all $k\ge k_1$. 
This, combined with Lemma~\ref{lm-basic}(iv), implies that 
$\{x^{k}\}$ converges finitely to a critical point of $\psi$, 
and that $\sum_{k=0}^{\infty}\|x^{k+1} - x^{k}\| < + \infty$.

On the other hand, if $\hat{\psi}_{k} > \overline{\psi}$ for all $k$, 
then it follows from \eqref{def-Delta}, \eqref{eq-sig} and \eqref{eq-sum-sig} that $\sum_{k=0}^{\infty}\|x^{k+1} - x^{k}\| < + \infty$. 
This implies that $\{x^{k}\}$ is convergent.
This, together with Lemma~\ref{lm-basic}(iv), completes the proof.
\end{proof}

\begin{lemma}\label{lm-rec}
Let $\psi:\mathbb{X}\to (-\infty, \infty]$ be a proper, closed, 
and $\varrho$-weakly convex function with constant $\varrho\ge 0$. 
Suppose that Property~\ref{ass-param} holds, 
and that $\psi$ is a KL function with exponent $\theta \in (0, 1)$.
Let $\{x^{k}\}$ satisfy {\rm (C\ref{C1})}, {\rm (C\ref{C2})} and {\rm (C\ref{C3})}.
Suppose that $\hat{\psi}_k > \overline{\psi}$ for all $k\ge 0$, 
where $\overline{\psi}$ is given in \eqref{eq-value-lim}.
Then for the $\tau$ in Property~\ref{ass-param}(iii), there exist $\eta>0$, $N \in \mathbb{N}$, $c_1>0$, $c_2>0$ and $\varphi(t) = \kappa t^{1-\theta_\tau}$ for some constant $\kappa > 0$, where $\theta_\tau := \max\{\theta, \frac{\tau - 1}{\tau}\}$, such that \eqref{eq-KL-ineq-2}, \eqref{eq-x-sig} and \eqref{eq-sum-sig} hold for all $k\ge N$. Moreover, there exist $N_2 \ge N$, $c_3 > 0$ and $c_4>0$ such that
\begin{align}
	& \sigma_{k+1} + c_3 \alpha_{k} 
	\sigma_{k+1}^{p} \le \sigma_{k}  + 
	c_4 \alpha_{k} (A_{k} + R_{k+1}^{p})
	\quad \forall k\ge N_2,
	\label{eq-sig-rec-2}
\end{align}
where $\{\sigma_l\}$ is given in \eqref{eq-sig}, $p \coloneqq \max\{1, \theta_{\tau}/(1- \theta_{\tau}) \}$, and 
\begin{equation}
	\label{eq-R}
	R_l \coloneqq  \sum_{j = l}^{\infty} \alpha_{j-1} A_{j-1}<\infty
	\quad \forall l \in \N.
\end{equation}
Furthermore, if $\theta \in (0, \frac{1}{2}]$ and $\tau = 2$, then 
\begin{align}\label{eq-rec-0-1/2}
	& \sigma_{k+1} \le (1+c_3 \alpha_k)^{-1} \sigma_{k}  + 
	c_4 \alpha_{k} (A_{k} +R_{k+1})
	\quad \forall k\ge N_2.
\end{align}
\end{lemma}
\begin{proof}
Since $\psi$ is a KL function with exponent $\theta \in (0, 1)$, 
by \cite[Proposition~4.2]{YKZH25}, we see that $\Psi_{\tau}$ defined in \eqref{eq-F-tau-def} is a KL function with exponent $\theta_\tau = \max\{\theta, \frac{\tau - 1}{\tau}\}$. 
By Lemma~\ref{lm-kl-x}, there exist $\eta>0$, $\varphi \in {\mathcal{Q}}_{\eta}$, $N \in \mathbb{N}$, $c_1>0$ and $c_2>0$ such that 
\eqref{eq-KL-ineq-2}, \eqref{eq-x-sig} and \eqref{eq-sum-sig} hold for all $k\ge N$, 
where the function $\varphi$ can be chosen as
$\varphi(t) = \kappa t^{1 - \theta_{\tau}}$
with some constant $\kappa>0$ in view of its proof; see \eqref{eq-kl-Psi} where the KL property was invoked for the existence of $\varphi$.	

By Lemma~\ref{lm-basic}(i) and (ii), and invoking the definition of $h_k$ in \eqref{eq-KL-ineq-2} and that of $\Delta_k$ in \eqref{def-Delta},
we see that there exists an integer $N_2\ge N$ such that 
\begin{equation}\label{eq-le-1}
	\varphi(h_k) \le 1\quad \mbox{and} \quad \Delta_k \le 1 
	\quad \forall k\ge N_2.
\end{equation}

Now, using \eqref{eq-KL-ineq-2} and $\varphi^{\prime}(t) = \kappa (1-\theta_{\tau}) t^{-\theta_{\tau}}$, 
we have for all $k\ge N_2$,
\begin{align}
\label{eq-tilde-c3}
\notag \Delta_{k} + \alpha_{k}A_{k}
& \ge \frac{\alpha_{k}} {c_1\kappa(1-\theta_{\tau})}  h_{k+1}^{\theta_{\tau}}
= \frac{\alpha_{k}} {c_1\kappa(1-\theta_{\tau})}  (\varphi(h_{k+1})/\kappa)^{\frac{\theta_{\tau}}{1-\theta_{\tau}}}
\\ & = \tilde{c}_3 \alpha_{k} \varphi(h_{k+1})^{\frac{\theta_{\tau}}{1-\theta_{\tau}}}
\ge \tilde{c}_3 \alpha_{k} \varphi(h_{k+1})^{p},
\end{align}
where $\tilde{c}_3\coloneqq 1/(c_1 \kappa^{\frac{1}{1-\theta_{\tau}}} (1-\theta_{\tau}))>0$,
the first equality holds since $\varphi(t) = \kappa t^{1 - \theta_{\tau}}$ for $t\ge 0$,
and the last inequality holds thanks to \eqref{eq-le-1} and the fact $p\ge \frac{\theta_{\tau}}{1-\theta_{\tau}}$.
Moreover, by definition, we have for all $k\ge N_2$,
\begin{align*}
\tilde{c}_3 \alpha_{k} \sigma_{k+1}^{p}
& = \tilde{c}_3 \alpha_{k} (\Delta_k + c_2 \varphi(h_{k+1}) + R_{k+1})^{p}	
\\ & \le 3^{p-1} \tilde{c}_3 \alpha_{k} \left({\Delta_k}^{p} + c_2^p \varphi(h_{k+1})^{p} + R_{k+1}^{p}	\right)
\\ & \le 3^{p-1} \left(\tilde{c}_3 \alpha_{k} {\Delta_k}^{p} + c_2^p
({\Delta_k} + \alpha_{k}A_{k})
+ \tilde{c}_3 \alpha_{k} R_{k+1}^{p}	\right)
\\ & \le 3^{p-1} \left( (\tilde{c}_3 \bar{\alpha} + c_2^p)
{\Delta_k} + c_2^p\alpha_{k}A_{k}
+ \tilde{c}_3 \alpha_{k} R_{k+1}^{p}	\right)
\\ & \le 3^{p-1} \left( (\tilde{c}_3 \bar{\alpha} + c_2^p)
(\sigma_{k} - \sigma_{k+1}) + c_2^p\alpha_{k}A_{k}
+ \tilde{c}_3 \alpha_{k} R_{k+1}^{p}	\right),
\end{align*}
where the first inequality due to convexity of $t\to t^p$ (since $p\ge 1$), 
the second inequality follows from \eqref{eq-tilde-c3}, 
the third inequality follows from \eqref{eq-alpha-bound} and the facts that
${\Delta_k} \le 1$ (see \eqref{eq-le-1}) and $p\ge 1$, 
and the last inequality holds because of \eqref{eq-x-sig}.
Rearranging terms of the last display, we then obtain \eqref{eq-sig-rec-2} with suitable constants $c_3>0$ and $c_4>0$.

Finally, we consider \eqref{eq-rec-0-1/2}.
Suppose that $\theta \in (0, \frac{1}{2}]$ and $\tau = 2$.
In this case, we have $\theta_{\tau} = \max\{\theta, \frac{\tau - 1}{\tau}\} = \frac{1}{2}$ and $p = \max\{1, \frac{\theta_{\tau}}{1 - \theta_{\tau}}\} = 1$.
It then follows from \eqref{eq-sig-rec-2} that
\begin{align*}
& \sigma_{k+1} + c_3 \alpha_{k} 
\sigma_{k+1} \le \sigma_{k}  + 
c_4 \alpha_{k} (A_{k} + R_{k+1})
\quad \forall k\ge N_2.
\end{align*}
This implies
\begin{align*}
\sigma_{k+1} & \le (1 + c_3 \alpha_{k})^{-1}\sigma_{k}  + 
\frac{c_4}{1 + c_3 \alpha_{k}}  \alpha_{k} (A_{k} + R_{k+1})
\\& \le  (1 + c_3 \alpha_{k})^{-1} \sigma_{k}  + 
c_4 \alpha_{k} (A_{k} + R_{k+1}).
\end{align*}
This shows \eqref{eq-rec-0-1/2}. The proof is complete.	
\end{proof}

Now, we are ready to prove Theorem~\ref{th-main}.
\begin{proof}[{Proof of Theorem~\ref{th-main}.}]
By Proposition~\ref{prop-paras}, Property~\ref{ass-param} holds under either condition {\rm (A\ref{(A)})} or {\rm (A\ref{(B)})}. 
Then, from Theorem~\ref{th-global-convergence}, item (i) follows.

Next, we turn to prove items (ii), (iii) and (iv), which look at the case when $\psi$ is a KL function with exponent $\theta\in (0,1)$.
If $\hat{\psi}_{k_0} = \overline{\psi}$ for some $k_0\ge 0$,	
then from \eqref{eq-f-beta-des} and \eqref{eq-value-lim}, it follows that
$\hat{\psi}_k = \overline{\psi}$ and $x^{k+1} = x^k$ for all $k\ge k_0$, 
which, together with Lemma~\ref{lm-basic}(iv), implies that $\{x^{k}\}$ converges finitely to $x^*$, and hence the conclusions hold trivially.

In what follows, we suppose that $\hat{\psi}_{k} > \overline{\psi}$ for all $k\ge 0$. 
Then the conclusions of Lemma~\ref{lm-rec} hold.
Note that for every $k\ge N_2$ and $k^{\prime} \ge 1$,
\begin{equation*}
	\| x^{k} - x^{k+k^{\prime} + 1} \| 
	\le \sum_{l=k}^{k+k^{\prime}} \| x^{l} - x^{l+1}\|
	\le \sum_{l=k}^{\infty} \| x^{l} - x^{l+1}\|
    \le \sum_{l=k}^{\infty} \Delta_l\le \sigma_{k},
\end{equation*}
where the third inequality follows from the definition of $\Delta_l$ in \eqref{def-Delta},
and the last inequality follows from \eqref{eq-sum-sig}.
Fixing $k\ge N_2$ and taking $k^{\prime} \to \infty$ yield
\begin{equation}\label{eq-xstar-sig}
	\|x^{k} - x^{*} \| \le \sum_{l=k}^{\infty} \| x^{l+1} - x^{l} \| \le \sigma_k.
\end{equation}

Now, we consider item (ii). 
Suppose further that $\psi$ is a KL function with an exponent $\theta \in (0, \frac{1}{2}]$ and condition {\rm (A\ref{(A)})} holds.
Then by Proposition~\ref{prop-paras}(i), one can set $\tau = 2$ in Property~\ref{ass-param}(iii), 
and it holds that $A_{k} = \mathcal{O}(q_{A}^{k})$ and $R_k = \mathcal{O}(q^{k}_{A})$ for some $q_A \in (0, 1)$, where $A_k$ and $R_k$ are given in \eqref{eq-sum-bdd} and \eqref{eq-R}, respectively.
Thus, there exist $M_1,M_2>0$ such that $A_k \le M_1 q_{A}^{k}$ and $R_{k} \le M_2 q_{A}^{k}$ for all $k\ge 1.$ 
Moreover, $\underline{\alpha} \coloneqq \inf\{\alpha_k\} > 0$ because
$\{\alpha_k\}$ is positive and $\alpha_k = \Theta(1)$.
Since $\theta \in (0, \frac{1}{2}]$ and $\tau = 2$, with the corresponding $N_2$, $c_3$ and $c_4$ in \eqref{eq-rec-0-1/2}, we have
for all $k\ge N_2$ that
\begin{align*}
	\sigma_{k+1} & \le (1 + c_3 \alpha_{k} )^{-1} \sigma_{k}  + 
	c_4 \alpha_{k} (A_{k} +R_{k+1})
	\le (1 + c_3 \underline{\alpha} )^{-1} \sigma_{k}  + 
	c_4 \bar{\alpha} (M_1  + M_2)q_{A}^{k}
	\\ & \le \widetilde{q} \sigma_{k} + M_3 \widetilde{q}^{k}
	\le \widetilde{q}^2 \sigma_{k-1} + 2 M_3 \widetilde{q}^{k}
	\le \widetilde{q}^{k+1-N_2} \sigma_{N_2} + M_3 (k+1-N_2)  \widetilde{q}^{k},
\end{align*}
where $\widetilde{q} \coloneqq \max\{(1 + c_3 \underline{\alpha})^{-1}, q_{A}\} \in (0, 1)$ and $M_3 \coloneqq c_4 \bar{\alpha} (M_1  + M_2)$.
Now, pick an integer $N_3\ge N_2$ such that 
$k+1-N_2\ge \frac{k}{2}$ and $(k+1-N_2) \widetilde{q}^{\frac{k}{2}} \le 1$ for all $k\ge N_3$.
Then we see from the above display that $\sigma_{k+1} \le (\sigma_{N_2} + M_3) \widetilde{q}^{\frac{k}{2}}$ for all $k\ge N_3$.
Combining this with \eqref{eq-xstar-sig}, we obtain \eqref{rate-main-R}.

Finally, we consider items (iii) and (iv). 
Suppose further that $\psi$ is a KL function with exponent $\theta \in (0, 1)$
and {\rm (A\ref{(B)})} holds.
Then, there exist $M_5>M_4>0$ such that 
\begin{equation}\label{boundalpha}
M_4(1+k)^{-r_{\alpha}} \le \alpha_k \le M_5 (1+k)^{-r_{\alpha}}\ \ \ \forall k\ge 0.
\end{equation}
By Proposition~\ref{prop-paras}(ii), for any $\tau \ge \widetilde{\tau}$, we have $A_{k} = \mathcal{O}(k^{-r_{\tau}})$ and 
$R_k = \mathcal{O}(k^{-r_{R}})$, where
$r_{R} \coloneqq r_{\alpha} + r_{\tau} - 1$, which is positive thanks to {\rm (A\ref{(B)})}.
This implies that
$$
B_{k} \coloneqq A_{k} + R_{k+1}^{p} = \mathcal{O}(k^{-r_B}),
$$
where $p := \max\{1,\frac{\theta_\tau}{1 - \theta_{\tau}}\}$ 
and $r_B \coloneqq \min\{r_{\tau}, p(r_{\alpha} + r_{\tau} - 1) \}.$
Thus, there exists $M_6>0$ such that $B_{k} \le M_6 (1+k)^{-r_B}$ for all $k\ge 0$.

Now, for each fixed $\tau \ge \widetilde\tau$,\footnote{Notice from Proposition~\ref{prop-paras}(ii) that (A2) implies Property~\ref{ass-param}(iii) holds with any $\tau \ge \widetilde\tau$.} with the corresponding $N_2$, $c_3$ and $c_4$ in \eqref{eq-sig-rec-2}, we have for all $k\ge N_2$ that
\begin{align*}
	& \sigma_{k+1} + c_3 M_4(1+k)^{-r_{\alpha}} \sigma_{k+1}^{p}  
	\le \sigma_{k+1} + c_3 \alpha_{k} \sigma_{k+1}^{p} 
	\\ & \le \sigma_{k}  + c_4 \alpha_{k} (A_{k} + R_{k+1}^{p})
	\le \sigma_{k}  + c_4 M_5  M_6 (1+k)^{-(r_{\alpha} + r_B)},
\end{align*}
where we use \eqref{boundalpha} in the first and third inequalities, and the second inequality holds thanks to \eqref{eq-sig-rec-2}.
This, together with Lemma~\ref{lm-rate}, implies
\begin{equation}\label{display_sigmak}
\sigma_k = \mathcal{O}(k^{-\widetilde{r}})
\quad \text{with} \quad 		
\widetilde{r} = \begin{cases}
	r_B & \mbox{if}\ p = 1,
	\\ \min\left\{ \frac{r_B}{p}, \frac{1-r_{\alpha}}{p - 1} \right\}
	& \mbox{if}\ p > 1.
\end{cases}
\end{equation}

If $\theta \in (0, \frac{1}{2}]$ and $\widetilde{\tau} = 2$, then we let $\tau = 2$, which implies that $\theta_\tau = \max\{\theta, \frac{\tau - 1}{\tau}\} = \frac{1}{2}$, 
$p = \max\{1, \frac{\theta_{\tau}}{1 - \theta_{\tau}}\} = 1$, 
and $r_B = \min\{r_{\tau}, p(r_{\alpha} + r_{\tau} - 1) \} = r_{\alpha} + r_{\tau} - 1$.
This, combined with \eqref{display_sigmak} and \eqref{eq-xstar-sig}, implies \eqref{rate-main1}.

On the other hand, if $\theta \in (\frac{1}{2}, 1)$ and $\tau\ge \widetilde{\tau}$, 
then $\theta_{\tau} =  \max\{\theta, \frac{\tau - 1}{\tau}\} > \frac{1}{2}$,
and $p = \max\{1, \frac{\theta_{\tau}}{1 - \theta_{\tau}}\} = \frac{\theta_{\tau}}{1 - \theta_{\tau}} > 1$.
Using this, the definition of $r_B$, \eqref{display_sigmak} and \eqref{eq-xstar-sig}, 
we obtain \eqref{rate-main2}.
This completes the proof.
\end{proof}

\appendix
\section{Proof of Proposition~\ref{prop-kl}}\label{AppendixA}

\begin{proof}
	By assumption, there exist $\delta_1>0$, $\eta_1>0$ and $\varphi_{1} \in \mathcal{Q}_{\eta_1}$ such that if 
	$$
	x\in \mathbb{B}(\bar{x}, \delta_1) \cap \dom \partial f
	\quad \text{and} \quad
	0<f(x) - f(\bar{x}) < \eta_1,
	$$ 
	one has
	\begin{equation}\label{eq-kl-f}
		\varphi_1^{\prime}(f(x) - f(\bar{x})) \dist(0, \partial f(x)) \ge 1.
	\end{equation}
	Moreover, there exist $\delta_2>0$, $\eta_2>0$ and $\varphi_{2} \in \mathcal{Q}_{\eta_2}$ such that if 
	$$
	y\in \mathbb{B}(\bar{y}, \delta_2) \cap \dom \partial g
	\quad \text{and} \quad
	0 < g(y) - g(\bar{y}) < \eta_2,
	$$ 
	one has
	\begin{equation}\label{eq-kl-g}
		\varphi_2^{\prime}(g(y) - g(\bar{y})) \dist(0, \partial g(y)) \ge 1.
	\end{equation}
	
	Define $\widetilde{\varphi}:[0, \eta)\to \mathbb{R}_{+}$ by
	$$
	\widetilde{\varphi}(t)\coloneqq 2 \varphi_1(\tfrac{t}{2}) + 2\varphi_2(\tfrac{t}{2})
	\quad \forall t\in [0, \eta),
	$$ 
	where $\eta \coloneqq \min\{\eta_1 / 2, \eta_2\}$.	
	Clearly, $\widetilde{\varphi} \in \mathcal{Q}_{\eta}$.

	Since $g$ is continuous at $\bar y$, there exists $\delta_3>0$ such that
	$|g(y)-g(\bar y)| < \eta$ for all $y\in \mathbb{B}(\bar{y}, \delta_3)$.
	Let $\delta \coloneqq \min\{\delta_1, \delta_2, \delta_3\}$.
	Suppose that 
	$$
	(x, y) \in  \mathbb{B}((\bar{x}, \bar{y}), \delta) \cap \dom \partial \Psi
	\quad \text{and} \quad 
	0 < \Psi(x, y) - \Psi(\bar{x}, \bar{y}) < \eta.
	$$
	We claim that $(\bar x,\bar y)\in \dom \partial \Psi$ and
	\begin{equation}\label{eq-kl-2}
		\widetilde{\varphi}^{\prime}\left(\Psi(x, y) - \Psi(\bar{x}, \bar{y})\right) 
		\dist \left(0, \partial \Psi(x, y) \right) \ge 1.
	\end{equation}
	
	Indeed, from \cite[Proposition 10.5]{RW98}, it follows that $\partial \Psi(u, v) = \partial f(u) \times \partial g(v)$ for any $(u,v)\in \dom \Psi$,\footnote{Although \cite[Proposition~10.5]{RW98} is stated under the additional assumptions
$\mathrm{d}f(u)(0)=0$ and $\mathrm{d}g(v)(0)=0$, these assumptions are not used in the proof of this subdifferential formula.} which implies that $(\bar x,\bar y)\in \dom \partial f \times \dom \partial g = \dom \partial \Psi$, $(x,y)\in  \dom \partial \Psi = \dom \partial f \times \dom \partial g$, and 
		$\dist \left(0, \partial \Psi(x, y) \right) \ge \max\{\dist(0, \partial f(x)), 
		\dist(0, \partial g(y))\}$.
	This, together with $\widetilde{\varphi}^{\prime}(t) = \varphi_1^{\prime}(\tfrac{t}{2}) + \varphi_2^{\prime}(\tfrac{t}{2})$ for $t\in (0, \eta)$, yields
	\begin{align}\label{eq-kl-Psi-lm}
		\notag \widetilde{\varphi}^{\prime}\left(\Psi(x, y) - \Psi(\bar{x}, \bar{y})\right) 
		\dist \left(0, \partial \Psi(x, y) \right)
		& \ge \varphi_1^{\prime} \left( \tfrac 1 2( \Psi(x, y) - \Psi(\bar{x}, \bar{y}) ) \right) \dist(0, \partial f(x))
		\\ & \ \  + \varphi_2^{\prime} \left( \tfrac 1 2( \Psi(x, y) - \Psi(\bar{x}, \bar{y}) ) \right)
		\dist(0, \partial g(y)).
	\end{align}
		
	\underline{Case 1}: $f(x) - f(\bar x) \le g(y) - g(\bar y)$.
	In this case, we have
	\begin{equation*}
		0< \Psi(x, y) - \Psi(\bar{x}, \bar{y})
		= f(x) - f(\bar x) + g(y) - g(\bar y) \le 2 ( g(y) - g(\bar y)) < 2 \eta \le 2 \eta_2.
	\end{equation*}
	Since $\varphi_2^{\prime}$ is nonincreasing on $(0, \eta_2)$, this implies
	\begin{equation*}
		\varphi_{2}^{\prime} \left( \tfrac{1}{2} ( \Psi(x, y) - \Psi(\bar{x}, \bar{y}) ) \right) \dist(0, \partial g(y))
		\ge \varphi_2^{\prime} \left(  g(y) - g(\bar y) \right)
		\dist(0, \partial g(y)) \ge 1,
	\end{equation*}		
	where the last inequality holds thanks to \eqref{eq-kl-g}.
	This, together with \eqref{eq-kl-Psi-lm}, implies \eqref{eq-kl-2}.
	
	\underline{Case 2}: $f(x) - f(\bar x) > g(y) - g(\bar y)$.
	In this case, we have
	\begin{align*}
		& 0< \Psi(x, y) - \Psi(\bar{x}, \bar{y})
		= f(x) - f(\bar x) + g(y) - g(\bar y) \le 2 \left( f(x) - f(\bar x) \right)
		\\ & = 2 ( \Psi(x, y) - \Psi(\bar{x}, \bar{y}) - (g(y) - g(\bar y)) ) < 4 \eta \le 2 \eta_1.
	\end{align*}
	Since $\varphi^{\prime}_1$ is nonincreasing on $(0, \eta_1)$, we have 
	\begin{equation*}
		\varphi_{1}^{\prime} \left( \tfrac{1}{2} ( \Psi(x, y) - \Psi(\bar{x}, \bar{y}) ) \right) \dist(0, \partial f(x))
		\ge \varphi_1^{\prime} \left(  f(x) - f(\bar{x}) \right)
		\dist(0, \partial f(x)) \ge 1,
	\end{equation*}	
	where the last inequality holds because of \eqref{eq-kl-f}.
	This, together with \eqref{eq-kl-Psi-lm}, implies \eqref{eq-kl-2}.
\end{proof}

\end{document}